\documentclass[bj,authoryear,noshowframe]{imsart}

\usepackage{amssymb,amsmath,bm}
\usepackage{tikz}
\usetikzlibrary{arrows.meta,positioning,matrix,fit}
\usepackage{float}
\usepackage{longtable}
\usepackage{tabularx}
\usepackage{multirow}
\usepackage{booktabs}
\usepackage{hyperref}
\usepackage{mathtools}
\usepackage{color}
\usepackage{xcolor}
\usepackage{mathrsfs}
\usepackage[version=4]{mhchem}
\usepackage{subfigure} 
\usepackage{dcolumn}% Align table columns on decimal point
\usepackage{ulem}
\usepackage{enumitem}
\usepackage{graphicx}
\usepackage[capitalise,nameinlink,noabbrev]{cleveref}
\DeclareMathOperator{\Dom}{Dom}
\DeclareMathOperator{\id}{id}

\startlocaldefs
\newtheorem{theorem}{Theorem}[section]
\newtheorem{assumptionx}{Assumption}[section]
\newtheorem{lemma}[theorem]{Lemma}
\newtheorem{corollary}[theorem]{Corollary}
\newtheorem{proposition}[theorem]{Proposition}
\newtheorem{definition}{Definition}[section]
\newtheorem{example}{Example}[section]
\newtheorem{remark}[theorem]{Remark}

\crefname{theorem}{Theorem}{Theorems}
\Crefname{theorem}{Theorem}{Theorems}

\crefname{assumptionx}{Assumption}{Assumptions}
\Crefname{assumptionx}{Assumption}{Assumptions}

\crefname{lemma}{Lemma}{Lemmas}
\Crefname{lemma}{Lemma}{Lemmas}

\crefname{corollary}{Corollary}{Corollaries}
\Crefname{corollary}{Corollary}{Corollaries}

\crefname{proposition}{Proposition}{Propositions}
\Crefname{proposition}{Proposition}{Propositions}

\crefname{definition}{Definition}{Definitions}
\Crefname{definition}{Definition}{Definitions}

\crefname{example}{Example}{Examples}
\Crefname{example}{Example}{Examples}

\crefname{remark}{Remark}{Remarks}
\Crefname{remark}{Remark}{Remarks}

\newcommand{\R}{\mathbb{R}}

\newcommand{\E}{\mathbb{E}}

\newcommand{\eps}{\varepsilon}

\newcommand{\Id}{\mathrm{Id}}

\newcommand{\Sym}{\mathrm{Sym}}

\newcommand{\diag}{\mathrm{diag}}
\newcommand{\tr}{\mathrm{tr}}
\newcommand{\supp}{\mathrm{supp}}
\newcommand{\F}{\mathcal{F}}
\newcommand{\G}{\mathcal{G}}

\endlocaldefs
\usepackage{chngcntr}

\newenvironment{suppinfo}
{%
  \clearpage
  \counterwithout{equation}{section}

  \setcounter{section}{0}%
  \setcounter{subsection}{0}%
  \setcounter{subsubsection}{0}%
  \setcounter{equation}{0}%
  \setcounter{figure}{0}%
  \setcounter{table}{0}%

  \renewcommand{\thesection}{S\arabic{section}}%
  \renewcommand{\thetable}{S\arabic{table}}%

}
{%
  \clearpage
}
\begin{document}

\begin{frontmatter}
\title{Entropy production under coarse-graining limits: Stationary preservation and nonstationary convergence}
%\title{A sample article title with some additional note\thanksref{T1}}
\runtitle{Entropy Production under Coarse-Graining}
%\thankstext{T1}{A sample of additional note to the title.}

\begin{aug}
%%%%%%%%%%%%%%%%%%%%%%%%%%%%%%%%%%%%%%%%%%%%%%%
%% Additional information such as            %%
%% identifying the corresponding author must %%
%% be included in in the Acknowledgments     %%
%% section if necessary.                     %%
\author[A]{%
  \inits{X.}%
  \fnms{Xinyu}~%
  \snm{Zhang}%
  %\ead[label=e1]{zhangxy353@gmail.com}%
}
\author[A]{%
  \inits{L.}%
  \fnms{Liu}~%
  \snm{Hong}%
  \ead[label=e1]{hongliu@sysu.edu.cn}%
}

%%%%%%%%%%%%%%%%%%%%%%%%%%%%%%%%%%%%%%%%%%%%%%%
%% Addresses                                 %%
%%%%%%%%%%%%%%%%%%%%%%%%%%%%%%%%%%%%%%%%%%%%%%%

\address[A]{%
  School of Mathematics,
  Sun Yat-sen University,
  Guangzhou 510275,
  P. R. China%
  \printead[presep={,\ }]{e1}%
}
\end{aug}

\begin{abstract}
In this paper, we examine the preservation and loss of entropy production under coarse-graining limits of singularly perturbed diffusion semigroups. Under the coefficient convergence assumptions, we establish the lower semicontinuity of the stationary rate \(\sigma_{\mathrm{ss}}\) and provide a necessary and sufficient condition for its convergence in terms of the weighted \(L^2\) approximation. In the nonstationary setting, by assuming the convergence of normalized semigroup orbits, the semigroup gradient estimate, uniform in the perturbation parameter, yields the convergence of \(\sigma_{\mathrm{ex}}\) at every positive time. Meanwhile, the lower semicontinuity and convergence results for the housekeeping entropy production \(\sigma_{\mathrm{hk}}\) follow by arguments analogous to those in the stationary case. Finally, the obtained results are applied to the slow–fast averaging and a stiff-potential limit as two concrete examples.
\end{abstract}


\end{frontmatter}
%\newpage
%\tableofcontents

\section{Introduction}
Singularly perturbed diffusions provide a general framework for deriving
effective descriptions of multiscale stochastic systems.
A small parameter in the generator typically controls the separation
between microscopic and macroscopic scales
\citep{Khasminskii1966SmallParameter}.
A coarse-graining map specifies the variables retained in the
macroscopic description.
For many classes of such systems, effective limiting equations
have been identified and justified through principles of averaging and
homogenization \citep{pavliotis2008multiscale}. More recent developments include quantitative averaging for
distribution-dependent diffusions
\citep{ChengHaoRockner2024Averaging,XuZhengMu2025MLE}
and uniform-in-time estimates for fully coupled slow--fast systems
\citep{CrisanEtAl2026Averaging}.

This paper focuses on how the entropy production rates (EPRs) of
microscopic diffusions compare with those of their effective limits.
Weak convergence of distributions alone does not ensure the convergence
of these rates, since they involve quadratic functionals of
probability currents and density gradients. EPR provides a complementary criterion
for assessing whether a reduced model preserves the dynamical
irreversibility. At the stationary state, nonzero probability currents can
persist even though the distribution no longer evolves.
EPR quantifies the resulting time-reversal asymmetry through the
relative entropy per unit time between forward and time-reversed
paths \citep{DaCostaPavliotis2023EntropyProduction}.
In the driven overdamped Langevin systems at fixed temperature, these
stationary currents are accompanied by the sustained heat dissipation:
EPR is proportional to its mean rate, and its nonnegativity expresses
the second law of thermodynamics\citep{seifert2012stochastic}.
Since the invariant measure alone does not determine these currents,
a reduced model may reproduce the stationary statistics of the
retained variables without preserving the full microscopic dissipation.

For the autonomous diffusions considered here, the nonstationary
total EPR can be decomposed into two nonnegative contributions:
the housekeeping (adiabatic) EPR $\sigma_{\mathrm{hk}}$ and the
excess (nonadiabatic) EPR $\sigma_{\mathrm{ex}}$
\citep{van2010three,Qian2013Decomposition}.
The housekeeping EPR measures the dissipation associated with the stationary
probability currents, weighted by the evolving relative density.
The excess EPR measures the relative-entropy dissipation during
the relaxation to the stationary states, a process that, for reversible diffusions,
can be described as a Wasserstein gradient flow of the relative entropy
in the transport metric determined by the diffusion matrix
\citep{Tschiderer2026Volatility}.
At the stationary states, the excess EPR vanishes and the housekeeping EPR
equals the stationary EPR $\sigma_{\mathrm{ss}}$.

We pursue two questions for diffusion semigroups that need not be
reversible, both at the stationary states and at positive times.
Which contributions to entropy production could be lost in the
singular coarse-graining limit?
Under what conditions do the microscopic EPRs converge to
their macroscopic counterparts?

For EPRs, formal asymptotic calculations identify losses
under coarse-graining: \citet{BoCelani2014JSP} derive anomalous
contributions through multiscale expansions for both discrete and continuous
slow--fast systems, while \citet{JiaChen2016PRE} uses quasi-steady-state
approximations for finite-state continuous-time Markov chains to derive
the nonincrease of housekeeping EPR and the preservation of excess EPR under
state elimination and aggregation. Complementing these formal asymptotic analyses, other studies establish
rigorous limit theorems and quantitative error estimates for
coarse graining. For reversible finite-state
slow--fast Markov chains, \citet{MielkeStephan2020EDP} establish
variational convergence of time-integrated dissipation functionals,
yielding convergence in the sense of the energy--dissipation principle
(EDP); along solutions, these functionals represent accumulated
nonstationary entropy production. \citet{GeJiaJin2021JSP} prove joint
weak limits of slow diffusions and accumulated entropy production or
housekeeping heat, allowing time-dependent coefficients and irreversible
coupled dynamics under reversibility of the frozen fast dynamics and
additional compatibility conditions. For nonreversible diffusions,
\citet{hartmann2020coarse} obtain relative entropy and Wasserstein
error estimates between projected and effective time marginals,
together with relative-entropy bounds on the path space.

Alongside these results on accumulated dissipation and errors
between probability laws, we study the limiting behaviour of
EPRs themselves. Under the stated coefficient convergence assumptions, the macroscopic
stationary EPR is bounded above by the lower limit of the microscopic
stationary EPRs. A weighted $L^2$ approximation condition, called
locking, is both necessary and sufficient for the convergence to the
macroscopic rate. For the nonstationary results, we additionally assume the convergence of
the microscopic semigroup orbits to the macroscopic orbits composed
with the coarse-graining map. Under the same coefficient convergence
assumptions, the housekeeping EPR is then lower semicontinuous at
positive times. Convergence to the macroscopic rate for every
admissible initial pair at every positive time is again equivalent
to locking.

The excess EPR depends on spatial gradients, which orbit convergence
alone can not control. Rather than establishing gradient convergence
first, we use a semigroup gradient estimate, uniform in the
perturbation parameter, to obtain the monotonicity of an exponentially
weighted excess EPR
\citep{bolley-gentil-phi-entropy,GentilZugmeyer2021Beckner}.
Combined with the convergence of time-integrated dissipation, this
monotonicity yields the locally uniform convergence of the excess EPR on
$(0,\infty)$. Neither locking nor a uniform bound on the initial
excess EPRs is required. In the nonlinear averaging setting with a
common invariant measure, this convergence in turn yields the strong
weighted $L^2$ convergence of the slow gradients of semigroup orbits
and their square roots; see \Cref{thm:L2-strong}. The conclusion illustrates how EPR analysis can yield additional convergence results for singularly perturbed parabolic equations.

%To illustrate the uniform gradient estimate, we give two additional sufficient criteria: an It\^o--Kunita derivative-flow criterion and a Ricci criterion based on \citet{ArnoldCarlenJu2008LargeTime}. Under the stated regularity assumptions, both yield a Bakry--\'Emery curvature lower bound uniform in the perturbation parameter.

%The examples cover the slow--fast averaging and a reversible stiff-potential limit. The Ornstein--Uhlenbeck (OU) model provides explicit matrix conditions for the convergence hypotheses. A nonlinear multiplicative-noise model with a common Gaussian invariant measure exhibits strictly positive housekeeping EPR losses at the stationary states and at every positive time. The stiff-potential example treats invariant measures concentrating on an affine constraint manifold and establishes a locally uniform convergence of the excess, hence total, EPR on $(0,\infty)$.

The whole paper is organized as follows: \Cref{sec:setting} introduces the basic framework and EPR functionals. \Cref{sec:main-results} presents the main results and additional sufficient criteria for justifying the corresponding hypotheses. \Cref{sec:case_study} develops the applications to two nontrivial examples: the slow--fast averaging and a reversible stiff-potential limit,
and \Cref{sec:discussion} discusses further questions.
The Supplementary Information provides auxiliary proofs
for the main results and applications.

\section{Setting}
\label{sec:setting}
Let $E\subset\mathbb R^N$ and $\bar E\subset\mathbb R^n$ be open,
connected state spaces, linked by a surjective $C^1$ map
$\Phi:E\to\bar E$ with $\operatorname{rank}(D\Phi)=n$.
We write $z\in E$ and $x=\Phi(z)\in\bar E$ for microscopic and
macroscopic states.

For $\eps>0$, we compare the microscopic Markov semigroup
$(P_t^\eps)_{t\ge0}$ on $E$ with the macroscopic Markov semigroup
$(\bar P_t)_{t\ge0}$ on $\bar E$.

\subsubsection*{Convention (generic notation)}
Symbols without superscripts $\varepsilon$ or overbars are \emph{generic} and refer to either the microscopic or macroscopic system.
We use $\nabla$ for column gradients of scalar functions,
$\nabla\!\cdot$ for divergence, and $D$ for derivatives, with
$(D\Psi)_{ij}=\partial_j\Psi_i$ for vector-valued maps.
Subscripts indicate the variables of differentiation.

% ----------------------------------------------------------------------
\subsection{Semigroup}
\label{subsec:set-semigroup}

Let $X$ denote either $E$ or $\bar E$, and let $(P_t)_{t\ge 0}$ be a Markov semigroup on $X$.
For a nonnegative initial datum $f_0$ we set
\[
u(t,\zeta):=(P_t f_0)(\zeta),\qquad t\ge 0,\ \zeta\in X.
\]
We assume that $P_t:C_b(X)\to C_b(X)$ and $\|P_t f_0\|_\infty\le \|f_0\|_\infty$ for all $t\ge 0$, and we identify $P_t f$ with its bounded continuous version whenever pointwise values are used.
Moreover, $(P_t)_{t\ge 0}$ is strongly continuous on $L^2(\pi)$ and admits an infinitesimal generator $\mathcal L$,
\begin{equation}
\mathcal L f=\lim_{t\downarrow 0}\frac{P_t f-f}{t}\quad\text{in }L^2(\pi),\qquad f\in\mathrm{Dom}(\mathcal L),
\label{eq:generator-eps}    
\end{equation}
together with a unique invariant probability measure $\pi$ with strictly positive Lebesgue density.
Finally, we assume that $C_c^\infty(X)$ is a core for $\mathcal L$. We write $\mathcal L^\dagger$ and $P_t^\dagger$ for the
$L^2(\pi)$ adjoints of $\mathcal L$ and $P_t$, respectively.
For $0\le f_0\in L^\infty(\pi)$ with $\int_X f_0\,d\pi=1$,
the orbit $u(t)=P_t f_0$ is the density with respect to $\pi$
of the time-$t$ marginal law of the process with transition
semigroup $(P_t^\dagger)_{t\ge0}$ and initial law $f_0\pi$.
For the process with transition semigroup $(P_t)_{t\ge0}$ and
the same initial law, the corresponding density is $P_t^\dagger f_0$.

The carr\'e-du-champ $\Gamma$ and its iterated form $\Gamma_2$ are defined on $C_c^\infty(X)$ by
\[
  \Gamma(f,g)
  :=\tfrac12\Bigl(\mathcal L(fg)-f\,\mathcal L g-g\,\mathcal L f\Bigr),
  \qquad
  \Gamma(f):=\Gamma(f,f),
\]
\[
  \Gamma_{2}(f)
  :=\tfrac12\Bigl(\mathcal L \Gamma(f)-2\,\Gamma\bigl(f,\mathcal L f\bigr)\Bigr).
\]

% ----------------------------------------------------------------------
\subsubsection{Diffusion-form specialization and dictionary}
\label{subsubsec:set-diffusion}

Here we restrict to diffusion semigroups: we write
\begin{equation}
\label{eq:stea-V}
\pi(d\zeta)=p\,d\zeta=Z^{-1}e^{-V(\zeta)}\,d\zeta    
\end{equation}

for some $V\in W_{loc}^{2,\infty}(X)$ and assume an irreversible drift $\gamma\in W_{loc}^{1,\infty}(X;\mathbb R^d)$ satisfying
\begin{equation}
\label{eq:stea-gamma}
\nabla\!\cdot(\gamma e^{-V})\equiv 0.
\end{equation}
We assume that $\mathcal L$ is of the diffusion type and that for every $\varphi\in C_c^\infty(X)$,
\begin{equation}
\mathcal L \varphi=\langle\gamma,\nabla \varphi\rangle+\frac{\nabla\!\cdot(p A\nabla \varphi)}{p},
\label{eq:backward-eps1}
\end{equation}
where $A\in W_{loc}^{2,\infty}(X;\R^{d\times d})$ is symmetric and locally uniformly positive definite (here $d=N$ on $E$ and $d=n$ on $\bar E$).
For every $f\in L^2(\pi)$, the curve $u(t):=P_t f$ belongs to $C([0,\infty);L^2(\pi))$ and is the unique mild solution of $\partial_t u=\mathcal L u$ with $u(0)=f$ in $L^2(\pi)$. If $f\ge0$ $\pi$-almost everywhere, then $u(t)\ge0$ $\pi$-almost everywhere for every $t\ge0$.

For diffusion generators, $\Gamma$ and $\Gamma_2$ are
understood through their local differential expressions,
almost everywhere under the stated Sobolev regularity.
\begin{lemma}[Diffusion dictionary]\label{lem:dict-diffusion}
Let $\mathcal L$ be given by \Cref{eq:backward-eps1}, and define
\begin{equation}
    \label{eq:LSA}
    \mathcal L^{s,a}:=\tfrac12(\mathcal L\pm\mathcal L^\dagger)
\end{equation}
 on
$\Dom(\mathcal L)\cap\Dom(\mathcal L^\dagger)$. Then
$C_c^\infty(X)$ is contained in this domain and, for $f,g\in C_c^\infty(X)$,
\[
\Gamma(f,g)=\langle\nabla f,A\nabla g\rangle,\qquad
\mathcal L^sf=p^{-1}\nabla\!\cdot(p A\nabla f),\qquad
\mathcal L^af=\gamma\cdot\nabla f.
\]
Moreover, $\Gamma(\id_i,\id_j)=A_{ij}$ and $\mathcal L^a\id=\gamma$ locally:
on $U\Subset X$, replace $\id$ by $\chi\id$, where
$\chi\in C_c^\infty(X)$ equals one near $\overline U$.
\end{lemma}

\begin{proof}
Writing $\mathcal S f:=p^{-1}\nabla\!\cdot(p A\nabla f)$ and
$\mathcal Bf:=\gamma\cdot\nabla f$, integration by parts and
$\nabla\!\cdot(p\gamma)=0$ give
\[
\langle f,\mathcal Lg\rangle_{L^2(\pi)}
=\langle(\mathcal S-\mathcal B)f,g\rangle_{L^2(\pi)}
\qquad(f,g\in C_c^\infty(X)).
\]
The core property extends this identity to $g\in\Dom(\mathcal L)$, hence
$\mathcal L^\dagger f=(\mathcal S-\mathcal B)f$; the remaining claims follow
from the diffusion product rule and locality.
\end{proof}

\subsubsection{Initial data}
We prescribe bounded positive profiles through their square roots.
Assume that the pre-forms
\[
\mathcal E_0^\varepsilon(g)
 :=\int_E\Gamma^\varepsilon(g)\,d\pi^\varepsilon,
 \quad g\in C_c^\infty(E),
\qquad
\bar{\mathcal E}_0(\bar g)
 :=\int_{\bar E}\bar\Gamma(\bar g)\,d\bar\pi,
 \quad \bar g\in C_c^\infty(\bar E),
\]
are closable (for every $\varepsilon$ in the first case), and denote
their closures by
$(\mathcal E^\varepsilon,\mathcal D(\mathcal E^\varepsilon))$ and
$(\bar{\mathcal E},\mathcal D(\bar{\mathcal E}))$. Set
\[
\begin{aligned}
\mathcal M^\varepsilon
&:=\bigl\{f=g^2:\,
g\in C_b(E)\cap C^\infty(E)\cap\mathcal D(\mathcal E^\varepsilon),
\ g>0\bigr\},\\
\bar{\mathcal M}
&:=\bigl\{\bar f=\bar g^2:\,
\bar g\in C_b(\bar E)\cap C^\infty(\bar E)
       \cap\mathcal D(\bar{\mathcal E}),\ \bar g>0\bigr\},
\qquad
\mathcal M:=\bigcap_{0<\varepsilon<\varepsilon_0}\mathcal M^\varepsilon .
\end{aligned}
\]
For the dynamical results, fix a microscopic profile class
$1\in\mathcal A\subset\mathcal M$.
Unless specified otherwise, $\mathcal A=\mathcal M$.
Each reference profile $f\in\mathcal A$ is fixed as $\eps$ varies.

Both profile classes contain the constant profile $1$.
Indeed, the operator-core identity
$\mathcal E_0(h)=-\langle h,\mathcal Lh\rangle_{L^2(\pi)}$
extends by graph-norm approximation to
$\operatorname{Dom}(\mathcal L)\subset \mathcal D(\mathcal E)$,
and conservativity gives $\mathcal L1=0$ and hence
$\mathcal E(1)=0$. For $f\in\mathcal A$ and $\bar f\in\bar{\mathcal M}$, set
\[
\begin{aligned}
&m^\varepsilon:=\int_Ef\,d\pi^\varepsilon,\qquad
f_0^\varepsilon:=\frac{f}{m^\varepsilon},\qquad
u^\varepsilon(t):=P_t^\varepsilon f_0^\varepsilon,\\
&\bar m:=\int_{\bar E}\bar f\,d\bar\pi,\qquad
\bar f_0:=\frac{\bar f}{\bar m},\qquad
\bar u(t):=\bar P_t\bar f_0 .
\end{aligned}
\]

All thermodynamic quantities below are evaluated along these normalized trajectories. Throughout the paper, all semigroups satisfy the standing assumptions above.

The square-root condition imposes finite initial energy for each
$\eps$, without requiring a uniform bound as $\eps\downarrow0$.
\subsection{Entropy production functionals}
\label{subsec:thermo-prelim}

%We impose the vanishing-boundary condition \Cref{ass:vanishing-boundary}, which removes cutoff boundary terms in integration by parts, on every microscopic and macroscopic semigroup considered below.

%\begin{corollary}\label{cor:cutoff-growth}Let $X=\mathbb R^d$ and\[\mathcal L\varphi=\operatorname{tr}(A\nabla^2\varphi)+b\cdot\nabla\varphi, \qquad b=\gamma+\operatorname{div}A-A\nabla V.\]If $|b(\zeta)|\le C(1+|\zeta|)$ and $\|A(\zeta)\|\le C(1+|\zeta|^2)$, then\Cref{ass:vanishing-boundary} holds.\end{corollary}

%\begin{proof}For standard radial cutoffs,\[|\mathcal L\chi_R|+\Gamma(\chi_R)\le C\mathbf 1_{\{R\le|\zeta|\le2R\}},\]whose $\pi$-integral tends to zero because $\pi$ is a probability measure.\end{proof}

\begin{definition}[Thermodynamic terms]\label{def:sigmas-single}
Let $f_0$ be one of the normalized initial profiles specified
above, and set $u_t:=P_tf_0$.  Then $\int_Xu_t\,d\pi=1$. With \(\mathcal L^a\) defined in \Cref{eq:LSA} and
\[
q:=(\mathcal L^a\mathrm{id})^\top\Gamma(\mathrm{id})^{-1}\mathcal L^a\mathrm{id}
  =\gamma^\top A^{-1}\gamma
\]
(the first expression being understood locally as above), define
\begin{align}
\sigma_{\mathrm{hk}}(t)
  &:= \int_X u_t q\,d\pi,&
  %\label{def:hk-single-abstract}
  %\label{eq:def_hkss}\\
\sigma_{\mathrm{ex}}(t)
  &:= 4\int_X \Gamma(\sqrt{u_t})\,d\pi,
  %\label{def:ex-single-abstract}
&\sigma(t)&:= \sigma_{\mathrm{hk}}(t)+\sigma_{\mathrm{ex}}(t).
  \label{def:epr-single-abstract}
\end{align}
All nonnegative integrals are understood in $[0,\infty]$. When \(f\equiv 1\), \(\sigma_{\mathrm{ex}}(t)\) reduces to zero, while \(\sigma_{\mathrm{hk}}\) and \(\sigma\) reduce to \begin{equation}\label{eq:def_hkss}
    \sigma_{\mathrm{ss}}:= \int_X q\,d\pi.
\end{equation}
\end{definition}
%Writing $d\pi=p\,d\zeta$ and $\rho_t=pu_t$, the $\pi$-adjoint diffusion has probability current$j_t=-p(u_t\gamma+A\nabla u_t)$.Under the cutoff condition in \Cref{ass:standing}, whenever both dissipation terms are finite, stationarity $\nabla\cdot(p\gamma)=0$ gives $\int_X\gamma\cdot\nabla u_t\,d\pi=0$, and hence\[\int_X\frac{j_t^\top A^{-1}j_t}{\rho_t}\,d\zeta=\sigma_{\mathrm{hk}}(t)+\sigma_{\mathrm{ex}}(t).\]This identifies their sum with the usual quadratic current dissipation \cite{QianQianTang2002JSP}; the cutoff argument is given in \Cref{sec:app-forward}.

%The housekeeping term measures dissipation associated with the stationary irreversible current, weighted by the evolving relative density, while the excess term describes dissipation associated with relaxation towards the invariant measure. In this autonomous setting, these are the adiabatic and nonadiabatic entropy production rates, respectively \cite{van2010three}.

We denote these functionals along $u^\eps(t)$ by $\sigma_{\mathrm{hk}}^\eps(t)$, $\sigma_{\mathrm{ex}}^\eps(t)$ and $\sigma^\eps(t)$, and use overbars for their counterparts along $\bar u(t)$.
\subsubsection{Forward representation}
For either system, $L^2(\pi)$ duality identifies $\rho_t:=pu_t$
as the forward density of the $\pi$-adjoint diffusion started
from $pf_0$. The divergence form of $\mathcal L$ and
$\nabla\cdot(p\gamma)=0$ give
\[
\partial_t\rho_t=-\nabla\cdot j_t,\qquad
j_t:=-p\bigl(A\nabla u_t+u_t\gamma\bigr),\qquad t>0,
\]
in distributions.

Suppose \Cref{ass:standing}(iii) holds and both dissipation
terms are finite at $t>0$. Then
$\gamma\cdot\nabla u_t\in L^1(\pi)$ by Cauchy--Schwarz inequality,
and integration by parts with the prescribed cutoffs gives
\[
\left|\int_X\chi_R\gamma\cdot\nabla u_t\,d\pi\right|
\le
\left(
\sigma_{\mathrm{hk}}(t)\|u_t\|_\infty
\int_X\Gamma(\chi_R)\,d\pi
\right)^{1/2}
\longrightarrow0.
\]
Dominated convergence yields
$\int_X\gamma\cdot\nabla u_t\,d\pi=0$, so expanding the square gives
\[
\int_X\frac{j_t^\top A^{-1}j_t}{\rho_t}\,d\zeta
=
\sigma_{\mathrm{hk}}(t)+\sigma_{\mathrm{ex}}(t)
=
\sigma(t).
\]
This is the housekeeping--excess 
decomposition \citep{van2010three}.
Under the cutoff condition in \Cref{ass:standing}(iii),
\Cref{lem:de-bruijn} gives the integrated entropy-dissipation
identity; under \Cref{ass:CDkappa}, \Cref{lem:I-continuity} gives
positive-time continuity of $\sigma_{\mathrm{ex}}$.

\begin{table}[H]
\centering
\small
\renewcommand{\arraystretch}{1.15}
\begin{tabularx}{\textwidth}{@{} l X @{}}
\toprule
\textbf{Symbol} & \textbf{Meaning} \\
\midrule

\multicolumn{2}{@{}l@{}}{\textbf{State spaces and coarse-graining}}\\
\midrule
$E,\ z$ & Microscopic state space and variable $z\in E\subset\mathbb R^N$.\\
$\bar E,\ x$ & Macroscopic state space and variable $x\in \bar E\subset\mathbb R^n$.\\
$X,\zeta$ & Underlying state space ($X=E$ or $X=\bar E$) and variable $\zeta\in X\subset\mathbb R^d$, $d=N$ or $n$.\\
$\Phi$ & Coarse-graining map; $x=\Phi(z)$ denotes the reduced variable.\\

\midrule
\multicolumn{2}{@{}l@{}}{\textbf{Semigroup objects}}\\
\midrule
$P_t$ & Markov semigroup on $X$; backward orbit $u(t,\cdot)=P_t f_0$.\\
$\pi$ & Invariant probability measure of $P_t$ (assumed to have a strictly positive Lebesgue density).\\
$\mathcal L$ & Generator of $P_t$ on $L^2(\pi)$.\\
$\Gamma(\phi,\psi)$ & Carr\'e du champ (defined via $\mathcal L$):
$\Gamma(\phi,\psi):=\tfrac12\big(\mathcal L(\phi\psi)-\phi\,\mathcal L\psi-\psi\,\mathcal L\phi\big)$;
$\ \Gamma(\phi):=\Gamma(\phi,\phi)$.\\
%$\Gamma_2(\phi)$ & Iterated carr\'e du champ (defined via $\mathcal L$ and $\Gamma$).\\
%$u(t,\cdot)$ & Backward orbit: $u(t,\cdot):=P_t f$.\\
$\mathcal M$ &
Common microscopic profile class with finite square-root form energy
for each $0<\eps<\eps_0$.
\\

$\mathcal A$ &
Microscopic profile class used for the dynamical results:
$1\in\mathcal A\subset\mathcal M$; by default,
$\mathcal A=\mathcal M$.
\\

$f$ &
Microscopic reference profile $f\in\mathcal A$, fixed as $\eps$ varies
(macro: $\bar f\in\bar{\mathcal M}$); the normalized initial data are
$f_0^\eps$ and $\bar f_0$, respectively.
\\

\midrule
\multicolumn{2}{@{}l@{}}{\textbf{Diffusion-form specialization (coordinate objects)}}\\
\midrule
$V$ & Potential of the invariant density: $\pi(d\zeta)=Z^{-1}e^{-V(\zeta)}\,d\zeta$.\\
$A(\cdot),\ \gamma(\cdot)$ & Diffusion matrix and nonreversible drift component in \Cref{eq:backward-eps1}.\\

\midrule
\multicolumn{2}{@{}l@{}}{\textbf{Thermodynamic functionals}}\\
\midrule
$\sigma_{\mathrm{ex}}(t),\sigma_{\mathrm{hk}}(t),\sigma(t)$ &  Entropy production rate: Excess , housekeeping, and total: $\sigma(t)=\sigma_{\mathrm{ex}}(t)+\sigma_{\mathrm{hk}}(t)$\\

\bottomrule
\end{tabularx}
\caption{\textbf{Notation table.}
For the singular perturbation family, microscopic objects carry a superscript $\varepsilon$
(e.g.\ $P_t^\varepsilon$ on $E$, $u^\varepsilon(t,\cdot)$),
while macroscopic limiting objects carry an overbar
(e.g.\ $\bar P_t$ on $\bar E$, $\bar u(t,\cdot)$).}
\label{tab:table}
\end{table}
\newpage
\section{Main Results}
\label{sec:main-results}
The logical structure of the main
results and the criteria for verifying the corresponding hypotheses are summarized in \Cref{fig:roadmap}.

\begin{figure}
    \centering
    \includegraphics[width=1\linewidth]{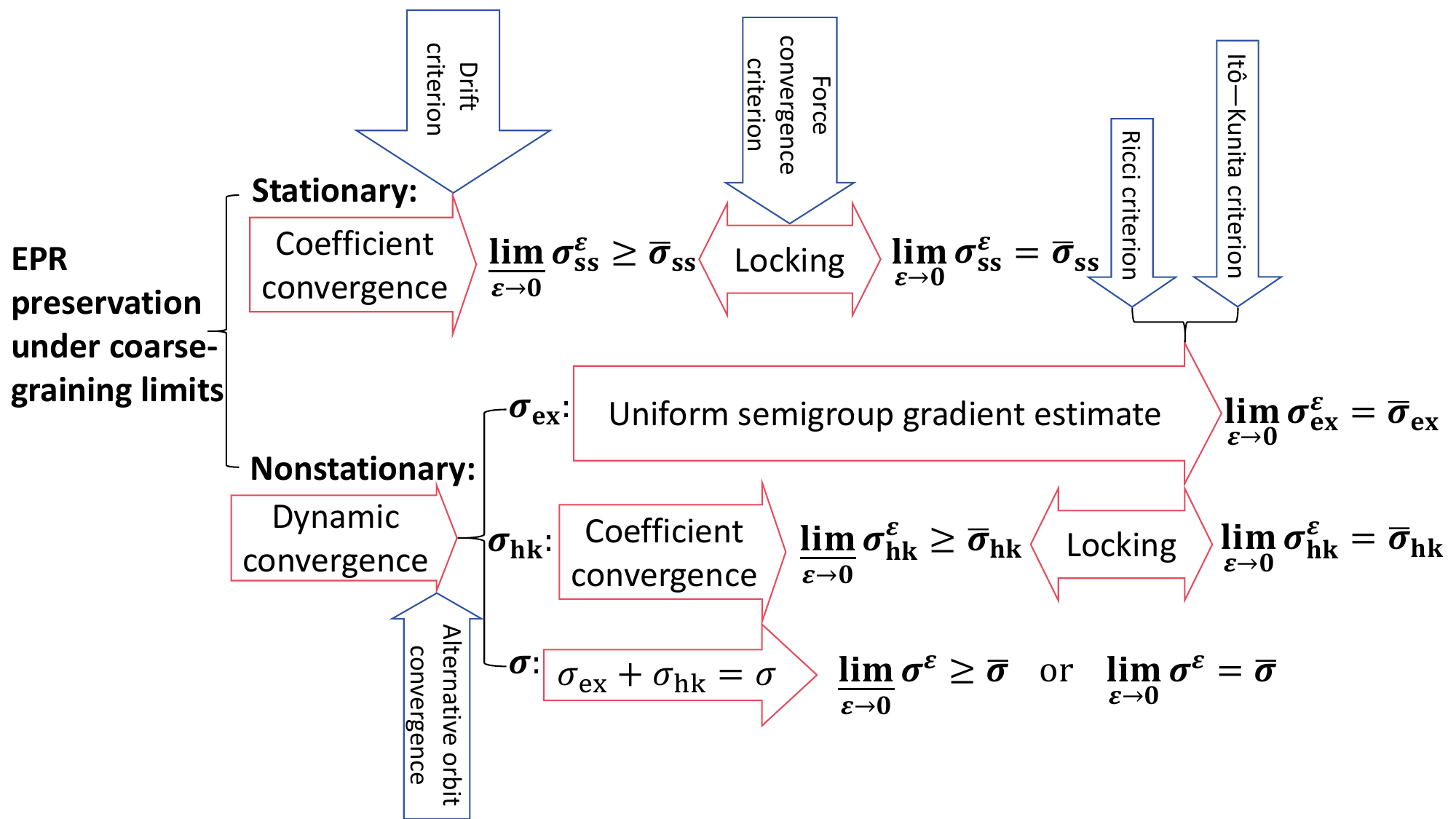}

    \caption{\textbf{Framework overview.}
Horizontal arrows indicate the main implications, while vertical
arrows identify verification criteria or alternative orbit conditions. %All implications are subject to the full hypotheses stated in the text. Nonstationary limits concern admissible initial pairs at positive times; the total EPR conclusions combine the excess and housekeeping results. The alternative orbit conditions are specified in \Cref{rem:dynamic}.
}
\label{fig:roadmap}
\end{figure}
\subsection{Stationary entropy production}
At the stationary states, the excess EPR vanishes.
Under the coefficient convergence condition (\Cref{ass:coeff-weak}), the stationary EPR is lower
semicontinuous, with convergence to the macroscopic EPR
characterized by the locking condition (\Cref{ass:locking}).
%The proof uses a quadratic variational representation whose nonnegative residual measures the unresolved dissipation.

Write $\mathcal M_{\mathrm{fin}}(X;\R^m)$ for the space of finite
vector-valued Radon measures, and
$\mathcal M_{\mathrm{fin}}(X;\mathbb S_+^m)$ for the cone of finite
positive semidefinite matrix-valued Radon measures.
\begin{assumptionx}[Coefficient convergence]\label{ass:coeff-weak}
%Under \Cref{ass:standing}, we assume the following.
\begin{enumerate}
\item[\textup{(i)}] \textbf{Weak convergence of projected current.} There exists a finite vector measure $\bar J=\bar\gamma\,\bar\pi
\in\mathcal M_{\mathrm{fin}}(\bar E;\R^n),$ such that for every $\xi\in C_b(\bar E;\mathbb{R}^n)$,
\begin{equation}\label{eq:J-weak-pushforward}
  \lim_{\varepsilon\to0}
  \int_E \big\langle \xi(\Phi(z)),\,D\Phi(z)\,\gamma^\varepsilon(z)\big\rangle\,d\pi^\varepsilon(z)
  = \int_{\bar E} \big\langle \xi(x),\,d\bar J(x)\big\rangle<\infty.
\end{equation}

\item[\textup{(ii)}] \textbf{Weak convergence of projected diffusivity.}
The measures
\[
Q^\eps:=(D\Phi A^\eps D\Phi^\top)\,\pi^\eps
\in\mathcal M_{\mathrm{fin}}(E;\mathbb S_+^n)
\]
are finite, and there exists
$Q\in\mathcal M_{\mathrm{fin}}(E;\mathbb S_+^n)$ such that
\begin{equation}\label{eq:Q-weak-pushforward}
Q^\eps\rightharpoonup Q,
\qquad
\Phi_\#Q=\bar Q:=\bar A\,\bar\pi
\in\mathcal M_{\mathrm{fin}}(\bar E;\mathbb S_+^n).
\end{equation}
Here weak convergence is tested against bounded continuous
symmetric matrix fields, using the pairing
$\langle\eta,Q\rangle:=\int_E\tr(\eta\,dQ)$.
\item[\textup{(iii)}] \textbf{Uniform projected microscopic housekeeping dissipation.}
For each $0<\varepsilon<\eps_0$ define the projected microscopic housekeeping dissipation by
\[
  \mathcal J_{\mathrm{hk,proj}}^\varepsilon
  := \int_E \big(D\Phi(z)\,\gamma^\varepsilon(z)\big)^\top
         \big(D\Phi(z)\,A^\varepsilon(z)\,D\Phi(z)^\top\big)^{-1}
         \big(D\Phi(z)\,\gamma^\varepsilon(z)\big)\,d\pi^\varepsilon(z).
\]
%where $\dagger$ denotes the Moore--Penrose pseudoinverse. 
We assume
\[
  \sup_{0<\varepsilon<\eps_0} \mathcal J_{\mathrm{hk,proj}}^\varepsilon < \infty.
\]
%Equivalently,\[D\Phi\,\gamma^\varepsilon \in L^2\!\big((D\Phi A^\varepsilon D\Phi^\top)^{-1};\,\pi^\varepsilon\big)  \quad\text{with}\quad  \sup_{0<\eps<\varepsilon_0} \left\|  D\Phi\,\gamma^\varepsilon \right\|_{L^2((D\Phi A^\varepsilon D\Phi^\top)^{-1};\,\pi^\varepsilon)}  <\infty.\]
\end{enumerate}
\end{assumptionx}

Clauses (i)--(ii) give the convergence of the linear and quadratic
terms in the variational representation. Clause (iii) bounds
the projected dissipation without requiring a uniform bound
on $\sigma^\eps_{\mathrm{ss}}$.

For affine $\Phi$, \Cref{prop:drift-implies-current} gives
a sufficient condition for \Cref{ass:coeff-weak}(i) in terms
of the projected drift.
Define the microscopic and macroscopic forces by
\begin{equation}\label{eq:F-thermoforce}
F^\eps := (A^\eps)^{-1}\gamma^\eps,
\qquad
\bar F := \bar A^{-1}\bar\gamma.
\end{equation}
For a symmetric positive definite matrix $B$, write
$|v|_B^2 := v^\top Bv$.

\begin{definition}[Recovery sequence]\label{def:recovery}
A sequence $(\psi_k)_{k\in\mathbb N}
\subset C_b(\bar E;\mathbb R^n)$
is called a recovery sequence if
\begin{equation}\label{eq:recovery}
\int_{\bar E}
|\psi_k(x)-\bar F(x)|_{\bar A(x)}^2
\,d\bar\pi(x)
\longrightarrow 0
\qquad (k\to\infty).
\end{equation}
\end{definition}

\begin{assumptionx}[Locking]\label{ass:locking}
There exists a recovery sequence $(\psi_k)$ such that
\begin{equation}\label{eq:locking-condition-thm}
\lim_{k\to\infty}\limsup_{\eps\downarrow0}
\int_E
\Big|
F^\eps(z)-D\Phi(z)^\top\psi_k(\Phi(z))
\Big|_{A^\eps(z)}^2
\,d\pi^\eps(z)=0.
\end{equation}
\end{assumptionx}
Intuitively, locking compares \(F^\varepsilon\) with the lift of the macroscopic force \(\bar F\). Since \(\bar F\) needs not be bounded, this comparison is formulated through bounded continuous recovery sequences converging to \(\bar F\). Under \Cref{ass:coeff-weak}, such recovery sequences exist by
\Cref{lem:hk_ss_finite}.

%\subsubsection{Results}
\begin{theorem}[Stationary EPR lower semicontinuity and locking]
\label{thm:ss-lsc-locking}
Suppose \Cref{ass:coeff-weak} holds. Then
$\bar\sigma_{\mathrm{ss}}<\infty$ and
\begin{equation}\label{eq:ss-ineq}
\liminf_{\eps\downarrow0}\sigma_{\mathrm{ss}}^\eps
\ge \bar\sigma_{\mathrm{ss}},
\end{equation}
with the convergence to the right-hand side if and only if
\Cref{ass:locking} holds.
\end{theorem}

\begin{proof}

By \Cref{lem:hk_ss_finite}, $\bar\sigma_{\mathrm{ss}}<\infty$.
For $\psi\in C_b(\bar E;\mathbb R^n)$, define
\[
\begin{aligned}
R_{\mathrm{ss}}^\eps(\psi)
&:=\int_E
\bigl|F^\eps-D\Phi^\top(\psi\circ\Phi)\bigr|_{A^\eps}^2
\,d\pi^\eps,\qquad
\delta(\psi)
:=\int_{\bar E}|\psi-\bar F|_{\bar A}^2\,d\bar\pi.
\end{aligned}
\]
%where $F^\eps$ and $\bar F$ are defined in \Cref{eq:F-thermoforce}.

The convergence established in \Cref{lem:hk_ss_finite} gives
\[
\widetilde\sigma^\eps(\psi)
=
\int_E
\left[
2\langle\psi\circ\Phi,D\Phi\gamma^\eps\rangle
-\bigl|D\Phi^\top(\psi\circ\Phi)\bigr|_{A^\eps}^2
\right]\,d\pi^\eps
\longrightarrow
\bar\sigma_{\mathrm{ss}}-\delta(\psi).
\]
Completing the square yields
\[
\sigma_{\mathrm{ss}}^\eps
=
\widetilde\sigma^\eps(\psi)+R_{\mathrm{ss}}^\eps(\psi)
\ge \widetilde\sigma^\eps(\psi).
\]
Hence
\[
\liminf_{\eps\downarrow0}\sigma_{\mathrm{ss}}^\eps
\ge \bar\sigma_{\mathrm{ss}}-\delta(\psi).
\]
Taking the supremum over $\psi$ and using
\Cref{eq:macro-hk-duality} proves \Cref{eq:ss-ineq}.

Recovery sequences exist by
\Cref{lem:hk_ss_finite}.
For any such sequence $(\psi_k)$, $\delta(\psi_k)\to0$,
so the same decomposition gives
\[
\lim_{k\to\infty}\limsup_{\eps\downarrow0}
R_{\mathrm{ss}}^\eps(\psi_k)
=
\limsup_{\eps\downarrow0}\sigma_{\mathrm{ss}}^\eps
-\bar\sigma_{\mathrm{ss}}.
\]
Thus \Cref{ass:locking} holds exactly when the right-hand side
vanishes, which, together with \Cref{eq:ss-ineq}, is equivalent
to the asserted convergence.
\end{proof}

With the stationary result established, we now give criteria for projected-current convergence and locking.

\begin{proposition}[Drift criterion for
\Cref{ass:coeff-weak}\textup{(i)}]
\label{prop:drift-implies-current}
Let $E=\R^N$, $\bar E=\R^n$, and $\Phi(z)=Mz+c$ with
$\operatorname{rank} M=n$. Suppose 
\Cref{ass:coeff-weak}\textup{(ii)--(iii)} hold. Define
\[
b^\eps:=\gamma^\eps+\nabla\!\cdot A^\eps-A^\eps\nabla V^\eps,
\qquad
\bar b:=\bar\gamma+\nabla\!\cdot\bar A-\bar A\nabla\bar V .
\]
Assume, for $0<\eps<\eps_0$,
\[
Mb^\eps\in L^1(\pi^\eps;\R^n),\qquad
\tr A^\eps\in L^1(\pi^\eps),
\]
and that the finite vector measures
\[
B^\eps:=\Phi_\#(Mb^\eps\,\pi^\eps)
\rightharpoonup \bar b\,\bar\pi
\quad\text{in }\mathcal M_{\mathrm{fin}}(\bar E;\R^n).\]
Then
\[J^\eps:=\Phi_\#(M\gamma^\eps\,\pi^\eps)
\rightharpoonup\bar\gamma\,\bar\pi
\quad\text{in }\mathcal M_{\mathrm{fin}}(\bar E;\R^n),
\]

and hence \Cref{ass:coeff-weak}\textup{(i)} holds.
\end{proposition}

\begin{proof}
Set $R^\eps:=MA^\eps M^\top$ and fix
$\xi\in C_c^1(\bar E;\R^n)$. Let $\chi_R$ be standard cutoffs with
$\chi_R=1$ on $B_R$, $\supp\chi_R\subset B_{2R}$, and
$|\nabla\chi_R|\le C/R$. Since
\[
\gamma^\eps\pi^\eps
=b^\eps\pi^\eps-\nabla\!\cdot(A^\eps\pi^\eps),
\]
testing against $\chi_RM^\top\xi(\Phi)$ gives

\begin{equation*}
\int\chi_R\langle\xi(\Phi),M\gamma^\eps\rangle\,d\pi^\eps
=\int\chi_R\langle\xi(\Phi),Mb^\eps\rangle\,d\pi^\eps
+\int\chi_R\tr\!\big((D\xi)(\Phi)R^\eps\big)\,d\pi^\eps
+\mathcal R_R^\eps ,
\end{equation*}

where
\[
|\mathcal R_R^\eps|
\le \frac{C\|\xi\|_\infty\|M\|}{R}
\int_{\{|z|\ge R\}}\tr A^\eps\,d\pi^\eps
\longrightarrow0 .
\]
Moreover, \Cref{ass:coeff-weak}\textup{(ii)--(iii)} imply
\[
\int|M\gamma^\eps|\,d\pi^\eps
\le
\big(\mathcal J^\eps_{\rm hk,proj}\big)^{1/2}
\left(\int\tr R^\eps\,d\pi^\eps\right)^{1/2}<\infty .
\]
Hence, letting $R\to\infty$,
\[
\langle J^\eps,\xi\rangle
=\langle B^\eps,\xi\rangle
+\langle\Phi_\#Q^\eps,D\xi\rangle .
\]
Passing to the limit using the assumptions yields
\[
\langle J^\eps,\xi\rangle
\longrightarrow
\int_{\bar E}\langle\xi,\bar b\rangle\,d\bar\pi
+\int_{\bar E}\tr((D\xi)\bar A)\,d\bar\pi
=
\int_{\bar E}\langle\xi,\bar\gamma\rangle\,d\bar\pi .
\]

Finally, for every Borel set $A\subset\bar E$,
\[
|J^\eps|(A)
\le
\big(\mathcal J^\eps_{\rm hk,proj}\big)^{1/2}
\big[\tr(\Phi_\#Q^\eps)(A)\big]^{1/2}.
\]
Thus \Cref{ass:coeff-weak}\textup{(ii)--(iii)} make
$J^\eps$ tight with uniformly bounded total variation.
Together with the preceding $C_c^1$ identification, this gives
$J^\eps\rightharpoonup\bar\gamma\,\bar\pi$.
\end{proof}

The following criterion tests the stationary locking condition directly against
the macroscopic force.

\begin{proposition}[Force convergence criterion for \Cref{ass:locking}]
\label{prop:direct-force-criterion}
Suppose \Cref{ass:coeff-weak} holds and
\begin{equation}
\label{eq:lifted-macro-force-energy}
\int_E
\bigl|D\Phi^\top(\bar F\circ\Phi)\bigr|_{A^\eps}^2
\,d\pi^\eps
\longrightarrow\bar\sigma_{\mathrm{ss}}.
\end{equation}
Then \Cref{ass:locking} holds if and only if
\begin{equation}
\label{eq:direct-lift-force-locking}
\int_E
\bigl|F^\eps-D\Phi^\top(\bar F\circ\Phi)\bigr|_{A^\eps}^2
\,d\pi^\eps\longrightarrow0.
\end{equation}
Condition \eqref{eq:lifted-macro-force-energy} is automatic when
$\bar F\in C_b(\bar E;\mathbb R^n)$.
\end{proposition}

\begin{proof}
Choose $\chi_k\in C_c(\bar E;[0,1])$ with $\chi_k\uparrow1$,
and set $\psi_k=\chi_k\bar F$.
By \eqref{eq:lifted-macro-force-energy} and
\Cref{ass:coeff-weak}(ii), we applied to the bounded continuous field
$(2\chi_k-\chi_k^2)\bar F\bar F^\top$, for each fixed $k$,
\[
\int_E
\bigl|D\Phi^\top((\bar F-\psi_k)\circ\Phi)\bigr|_{A^\eps}^2
\,d\pi^\eps
\longrightarrow
\delta_k:=\int_{\bar E}|\bar F-\psi_k|_{\bar A}^2\,d\bar\pi,
\]
where $\delta_k\to0$ as $k\to\infty$.
If \Cref{ass:locking} holds, \Cref{thm:ss-lsc-locking} and
the completion-of-squares identity in its proof give
$R_{\mathrm{ss}}^\eps(\psi_k)\to\delta_k$.
The triangle inequality and $\delta_k\to0$ yield
\eqref{eq:direct-lift-force-locking}.
Conversely, \eqref{eq:lifted-macro-force-energy},
\eqref{eq:direct-lift-force-locking} and the reverse triangle
inequality give $\sigma_{\mathrm{ss}}^\eps\to\bar\sigma_{\mathrm{ss}}$,
so \Cref{thm:ss-lsc-locking} yields \Cref{ass:locking}.
The bounded-force assertion follows from \Cref{ass:coeff-weak}(ii).
\end{proof}
\subsection{Time-dependent entropy production}
\label{subsec:dyn-con}
We now compare the two EPR components along the normalized
trajectories at positive times.
The assumptions below specify the dynamical input and the gradient
estimate used for the excess component.

\begin{assumptionx}[Dynamic convergence]\label{ass:standing}
%We assume the following.
\begin{enumerate}
\item[\textup{(i)}] There exists a Borel probability measure $\Pi$ on $E$ such that
\[
  \pi^\varepsilon \xrightarrow[\varepsilon\to0]{w} \Pi,
  \qquad \Phi_{\#}\Pi=\bar\pi,
\]i.e.\ $\int_E \varphi\,d\pi^\varepsilon\to\int_E \varphi\,d\Pi$ for all $\varphi\in C_b(E)$ and $\int_E \psi\!\circ\!\Phi\,d\Pi=\int_{\bar E}\psi\,d\bar\pi$ for all $\psi\in C_b(\bar E)$.

\item[\textup{(ii)}] For every $f\in\mathcal A$, there exists
$\bar f\in\bar{\mathcal M}$ such that, for every $t>0$ and every
compact set $K\subset E$,
\begin{equation}\label{eq:dyn-conv-standing}
  \lim_{\varepsilon\to0}\,
  \sup_{z\in K}\big|P_t^\varepsilon f^\varepsilon_0(z)-\bar P_t \bar f_0\big(\Phi(z)\big)\big|=0.
\end{equation}

\item[\textup{(iii)}] %[Vanishing boundary contribution]
There exist $\chi_R\in C_c^\infty(X)$ such that
\[
0\le\chi_R\le1,\qquad \chi_R\uparrow1,\qquad
\int_X\bigl(|\mathcal L\chi_R|+\Gamma(\chi_R)\bigr)\,d\pi\longrightarrow0.
\]
\end{enumerate}
\end{assumptionx}
Alternative orbit convergence for \Cref{ass:standing}(ii) is given in \Cref{rem:dynamic}.
\begin{definition}[Admissible initial pair]
\label{def:admissible-pair}
A pair $(f,\bar f)\in\mathcal A\times\bar{\mathcal M}$ is admissible
if its normalized trajectories satisfy
\Cref{ass:standing}(ii).
When that clause is replaced by \Cref{eq:orbit-L1} or
\Cref{eq:orbit-weighted}, admissibility is defined using the
chosen replacement.
\end{definition}
The constant pair $(1,1)$ is admissible under each of these orbit
conditions by conservativity.
%All thermodynamic quantities below are evaluated along the normalized trajectories associated with admissible pairs. Throughout the paper, all semigroups satisfy the standing assumptions above.
\begin{assumptionx}[Uniform semigroup gradient estimate]
\label{ass:CDkappa}
There exist $\eps_0>0$ and $\kappa\ge0$ such that, for
$0<\eps<\eps_0$,
\[
\sqrt{\Gamma^\eps(P_t^\eps f)}
 \le e^{\kappa t}P_t^\eps\sqrt{\Gamma^\eps(f)},
\qquad t>0,\quad f\in C_c^\infty(E),
\]
and
\[
\sqrt{\bar\Gamma(\bar P_t g)}
 \le e^{\kappa t}\bar P_t\sqrt{\bar\Gamma(g)},
\qquad t>0,\quad g\in C_c^\infty(\bar E).
\]
The extensions to bounded functions in the respective closed form
domains follow from \Cref{lem:gradient-form-extension}.
\end{assumptionx}
This condition permits gradient growth at rate $e^{\kappa t}$;
the essential requirement is that the same bound holds uniformly
in $\eps$. 

\Cref{thm:flow-CD,thm:Ricci} provide two $\varepsilon$--uniform sufficient criteria:
an It\^o--Kunita derivative-flow criterion and a Ricci-type matrix test.
%\subsubsection{Results}

\begin{theorem}[Nonstationary entropy production]
\label{thm:FI-conv}
Suppose \Cref{ass:standing} holds.
For every admissible pair
$(f,\bar f)\in\mathcal A\times\bar{\mathcal M}$ and every $t>0$:
\begin{enumerate}
\item If \Cref{ass:CDkappa} holds, then
\begin{equation}\label{eq:ex-eq}
\sigma_{\mathrm{ex}}^\eps(t)
\longrightarrow\bar\sigma_{\mathrm{ex}}(t)
\qquad(\eps\downarrow0).
\end{equation}

\item If \Cref{ass:coeff-weak} holds, then
\begin{equation}\label{eq:lsc-ineq}
\liminf_{\eps\downarrow0}\sigma_{\mathrm{hk}}^\eps(t)
\ge\bar\sigma_{\mathrm{hk}}(t).
\end{equation}
\end{enumerate}
Under \Cref{ass:coeff-weak}, convergence in the second assertion
holds for every admissible pair at every positive time
if and only if \Cref{ass:locking} holds.
\end{theorem}

\begin{proof}
Define the microscopic relative entropy \(\F^\eps\) and macroscopic relative entropy \(\bar \F\):
\begin{equation}
    \F^\eps(t):=\int u^\eps(t)\log u^\eps(t) d\pi^\eps,\qquad \bar\F(t):=\int \bar u(t)\log \bar u(t) d\bar \pi.
    \label{eq:F-def}
\end{equation}
\Cref{lem:de-bruijn} gives
\[
\F^\eps(b)-\F^\eps(a)
=-\int_a^b\sigma_{\mathrm{ex}}^\eps(s)\,ds,
\qquad 0<a<b,
\]
while \Cref{lem:free-energy-conv} yields
$\F^\eps(t)\to\bar\F(t)$ for every $t>0$.
By \Cref{lem:I-continuity},
$\bar\sigma_{\mathrm{ex}}$ is continuous on $(0,\infty)$,
so the macroscopic version of the same identity gives
$\bar\F\in C^1((0,\infty))$ and
$\bar\F'(t)=-\bar\sigma_{\mathrm{ex}}(t)$.
Moreover, by \Cref{lem:M-invariant}, the function
$\mathcal G^\eps(t):=e^{-2\kappa t}\sigma_{\mathrm{ex}}^\eps(t)$
is nonincreasing on $(0,\infty)$. Hence, for $0<h<t$,
\[
\frac{\F^\eps(t)-\F^\eps(t+h)}
     {\int_t^{t+h}e^{2\kappa s}\,ds}
\le \G^\eps(t)\le
\frac{\F^\eps(t-h)-\F^\eps(t)}
     {\int_{t-h}^{t}e^{2\kappa s}\,ds}.
\]
Letting first $\eps\downarrow0$ and then $h\downarrow0$ yields
$\G^\eps(t)\to e^{-2\kappa t}\bar\sigma_{\mathrm{ex}}(t):=\bar\G(t)$,
proving \Cref{eq:ex-eq}.

For the second assertion, fix $t>0$.
By \Cref{lem:hk_ss_finite},
\[
\bar\sigma_{\mathrm{hk}}(t)
\le\|\bar f_0\|_\infty\bar\sigma_{\mathrm{ss}}<\infty.
\]
For $\psi\in C_b(\bar E;\mathbb R^n)$, put
\[
R_{f,t}^\eps(\psi):=
\int_Eu^\eps(t)|F^\eps-D\Phi^\top(\psi\circ\Phi)|_{A^\eps}^2
\,d\pi^\eps,\qquad
\delta_t(\psi):=
\int_{\bar E}\bar u(t)|\psi-\bar F|_{\bar A}^2\,d\bar\pi,
\]
where \(F^\eps,\bar F\)  are defined in \Cref{eq:F-thermoforce}. Define the finite quadratic functional
\[
\widetilde\sigma_t^\eps(\psi)
:=\int_E u^\eps(t)\left[
2\langle\psi\circ\Phi,D\Phi\gamma^\eps\rangle
-\big|D\Phi^\top(\psi\circ\Phi)\big|_{A^\eps}^2
\right]\,d\pi^\eps.
\]
Completion of squares gives the exact identity
\[
\sigma_{\mathrm{hk}}^\eps(t)
=R_{f,t}^\eps(\psi)+\widetilde\sigma_t^\eps(\psi).
\]
By \Cref{lem:weighted-cross-quadratic-conv}, for every fixed $\psi$,
\[
\widetilde\sigma_t^\eps(\psi)
\longrightarrow
\int_{\bar E}\bar u(t)
\left[2\langle\psi,\bar\gamma\rangle-|\psi|_{\bar A}^2\right]
\,d\bar\pi
=\bar\sigma_{\mathrm{hk}}(t)-\delta_t(\psi).
\]
Choose a recovery sequence $(\psi_k)$ using
\Cref{lem:hk_ss_finite}. Then
\[
\delta_t(\psi_k)
\le \|\bar f_0\|_\infty
\int_{\bar E}|\psi_k-\bar F|_{\bar A}^2\,d\bar\pi
\longrightarrow0.
\]
Since $R_{f,t}^\eps(\psi_k)\ge0$, for each fixed $k$,
\[
\liminf_{\eps\downarrow0}\sigma_{\mathrm{hk}}^\eps(t)
\ge \lim_{\eps\downarrow0}\widetilde\sigma_t^\eps(\psi_k)
=\bar\sigma_{\mathrm{hk}}(t)-\delta_t(\psi_k).
\]
Letting $k\to\infty$ proves \Cref{eq:lsc-ineq}.

Assume \Cref{ass:locking}, and choose a stationary recovery
sequence $(\psi_k)$ realizing this condition.
Since $m^\eps\to\int_E f\,d\Pi>0$, there exist
$\eps_1>0$ and $C_f<\infty$ such that
$\|u^\eps(t)\|_\infty\le C_f$ for $0<\eps<\eps_1$.
Consequently,
\[
R_{f,t}^\eps(\psi_k)\le C_f R_{\mathrm{ss}}^\eps(\psi_k).
\]
The preceding decomposition and $\delta_t(\psi_k)\ge0$ give
\[
\limsup_{\eps\downarrow0}\sigma_{\mathrm{hk}}^\eps(t)
\le
\bar\sigma_{\mathrm{hk}}(t)
+C_f\limsup_{\eps\downarrow0}R_{\mathrm{ss}}^\eps(\psi_k).
\]
Letting $k\to\infty$ and combining with the lower bound proves
convergence. Conversely, convergence for every admissible pair
applies to $(f,\bar f)=(1,1)$ and therefore implies
\Cref{ass:locking} by \Cref{thm:ss-lsc-locking}.
\end{proof}

\begin{corollary}[Uniform convergence and total entropy production]
\label{cor:F-uniform}
Suppose \Cref{ass:standing,ass:CDkappa} hold, and let
$\F^\eps,\bar\F$ be the relative entropies introduced
in  \Cref{eq:F-def}.
\begin{enumerate}
\item For every admissible pair and every $0<\tau<T<\infty$,
\[
\sup_{t\in[\tau,T]}
\left(
|\F^\eps(t)-\bar\F(t)|
+|\sigma_{\mathrm{ex}}^\eps(t)-\bar\sigma_{\mathrm{ex}}(t)|
\right)\longrightarrow0.
\]

\item If \Cref{ass:coeff-weak} also holds, then, for every
admissible pair and every $t>0$,
\[
\liminf_{\eps\downarrow0}\sigma^\eps(t)\ge\bar\sigma(t).
\]
Moreover, convergence holds for every admissible pair at every
positive time if and only if \Cref{ass:locking} holds.
\end{enumerate}
\end{corollary}

\begin{proof}
Fix an admissible pair and $0<\tau<T<\infty$.
By \Cref{thm:FI-conv} and the auxiliary results used in its proof,
$\F^\eps$ and
$\G^\eps(t)=e^{-2\kappa t}\sigma_{\mathrm{ex}}^\eps(t)$
converge pointwise to the continuous functions $\bar\F$ and
$\bar \G(t)=e^{-2\kappa t}\bar\sigma_{\mathrm{ex}}(t)$.
Both $\F^\eps$ and $\G^\eps$ are nonincreasing, so a finite-partition
argument yields uniform convergence on $[\tau,T]$.
Multiplication by $e^{2\kappa t}$ gives the first assertion.

Under \Cref{ass:coeff-weak}, the second assertion follows from
$\sigma^\eps=\sigma_{\mathrm{hk}}^\eps+\sigma_{\mathrm{ex}}^\eps$,
the analogous macroscopic decomposition, and
\Cref{thm:FI-conv}, since
$\sigma_{\mathrm{ex}}^\eps(t)\to\bar\sigma_{\mathrm{ex}}(t)<\infty$.
\end{proof}

We next discuss an alternative orbit condition and give criteria for the uniform semigroup gradient estimate.

\begin{remark}[Alternative orbit convergence for \Cref{ass:standing}(ii)]
\label{rem:dynamic}
With other assumptions unchanged, \Cref{ass:standing}(ii)
may be weakened as follows, in each case requiring that for every
$f\in\mathcal A$ there exists $\bar f\in\bar{\mathcal M}$
whose normalized trajectories satisfy the stated condition.

For the excess EPR convergence conclusions,
including locally uniform convergence on $(0,\infty)$,
it suffices to require
\begin{equation}\label{eq:orbit-L1}
\int_E
\left|u^\eps(t)-\bar u(t)\circ\Phi\right|
\,d\pi^\eps
\longrightarrow 0,
\qquad t>0.
\end{equation}
This is the only consequence of \Cref{ass:standing}(ii) used
in the relative-entropy argument; the remaining
excess-convergence arguments are unchanged.

For the housekeeping conclusions, it suffices instead to require
\begin{equation}\label{eq:orbit-weighted}
\int_E
\left|u^\eps(t)-\bar u(t)\circ\Phi\right|^2
\operatorname{tr}\bigl(D\Phi A^\eps D\Phi^\top\bigr)
\,d\pi^\eps
\longrightarrow 0,
\qquad t>0.
\end{equation}
Under \Cref{ass:coeff-weak}, this yields \Cref{eq:L2-alter}
by Cauchy--Schwarz inequality, so the proof of
\Cref{lem:weighted-cross-quadratic-conv} and the subsequent
housekeeping arguments remain valid.
For conclusions involving both EPR components, we require
\Cref{eq:orbit-L1,eq:orbit-weighted} for the same initial pair.
\end{remark}

We write $CD(-\kappa,\infty)$ for the inequality
\[
\Gamma_2(\varphi)\ge-\kappa\Gamma(\varphi),
\qquad \varphi\in C_c^\infty(X).
\]

Both criteria below are stated under the lifetime and orbit-regularity
hypotheses of \Cref{lem:curvature-to-gradient}, imposed on each microscopic semigroup
and on the macroscopic semigroup. Under these hypotheses, a common
curvature lower bound yields \Cref{ass:CDkappa}; the auxiliary regularity
bounds need not be uniform in $\varepsilon$.
\subsubsection*{The It\^o--Kunita derivative-flow criterion}

For each $\eps>0$, write the microscopic generator as
\[
\mathcal L^\eps\varphi
=b^\eps\cdot\nabla\varphi+A^\eps:D^2\varphi,
\]
and set $\eta^\eps:=(2A^\eps)^{1/2}$ and
$G^\eps:=(A^\eps)^{-1}$.
Its SDE representation is
\[
dZ_t^\eps=b^\eps(Z_t^\eps)\,dt
+\eta^\eps(Z_t^\eps)\,dW_t.
\]
For the following theorem, assume that $A^\eps$ is positive definite,
$b^\eps$ is $C^2$ with bounded first and second derivatives, and
$\eta^\eps\in C_b^2$.
Then the SDE generates a non-explosive $C^1$ flow
$Z_t^{\eps,z}$ with Jacobian
$J_t^\eps(z):=D_zZ_t^{\eps,z}$; see
\cite[Chapter~4]{kunita1997stochastic}.

\begin{theorem}[Derivative-flow criterion for uniform gradient estimate]
\label{thm:flow-CD}
Under the common hypotheses above and the preceding coefficient
assumptions, suppose that, for some $\varepsilon_0>0$ and
$\rho\in\mathbb R$,
\begin{equation}
\label{eq:B2-Jacobian-energy}
\mathbb E\!\left[
(J_t^\eps(z)v)^\top G^\eps(Z_t^{\eps,z})J_t^\eps(z)v
\right]
\le e^{2\rho t}v^\top G^\eps(z)v
\end{equation}
for $0<\eps\le\eps_0$, $t\ge0$, $z\in E$ and
$v\in\mathbb R^N$.
Then $\mathcal L^\eps$ satisfies $CD(-\rho,\infty)$ and, with
$\kappa_{\rm mic}:=\rho_+$,
\begin{align}
\Gamma^\eps(P_t^\eps\varphi)
&\le e^{2\rho t}P_t^\eps\Gamma^\eps(\varphi),
\label{eq:B2-gradient-L2}\\
\sqrt{\Gamma^\eps(P_t^\eps\varphi)}
&\le e^{\kappa_{\rm mic}t}P_t^\eps
\sqrt{\Gamma^\eps(\varphi)}
\label{eq:B2-gradient-L1}
\end{align}
for $\varphi\in C_c^\infty(E)$.
Both estimates extend $\pi^\eps$-almost everywhere to
$\mathcal D(\mathcal E^\eps)\cap L^\infty(\pi^\eps)$.
Thus the microscopic clause of \Cref{ass:CDkappa} holds.

If the macroscopic diffusion has a $C^1$ flow $\bar Z_t^x$ satisfying,
with $\bar J_t(x):=D_x\bar Z_t^x$ and $\bar G:=\bar A^{-1}$,
\[
\mathbb E\!\left[
(\bar J_t(x)w)^\top\bar G(\bar Z_t^x)\bar J_t(x)w
\right]
\le e^{2\bar\rho t}w^\top\bar G(x)w
\]
for some $\bar\rho\in\mathbb R$ and all $t\ge0$, $x\in\bar E$ and
$w\in\mathbb R^n$, then the full \Cref{ass:CDkappa} holds with
$\kappa=\max\{\rho_+,\bar\rho_+\}$.
\end{theorem}

\begin{proof}
Flow differentiation
\cite[Section~2]{elworthy-li-heat-derivatives} and weighted
Cauchy--Schwarz inequality give
\[
|D(P_t^\eps\varphi)(z)v|^2
\le e^{2\rho t}v^\top G^\eps(z)v\,
P_t^\eps\Gamma^\eps(\varphi)(z).
\]
Taking the supremum over $v^\top G^\eps(z)v=1$ proves
\Cref{eq:B2-gradient-L2}.
Differentiation at $t=0$ gives
$\Gamma_2^\eps\ge-\rho\Gamma^\eps$.
Applying \Cref{lem:curvature-to-gradient} with
$\kappa=\kappa_{\mathrm{mic}}$
gives \Cref{eq:B2-gradient-L1}.
The form-domain extensions follow from
\Cref{lem:gradient-form-extension}.
The macroscopic assertion follows identically.
\end{proof}
\subsubsection*{Ricci criterion}

For a square matrix $M$, write
$\operatorname{Sym}(M):=(M+M^\top)/2$.
In the smooth diffusion setting of
\citet[Eqs.~(2.12)--(2.13)]{ArnoldCarlenJu2008LargeTime},
set $F:=A^{-1}\gamma$ and denote the symmetric Ricci matrix by
$\mathfrak R_z(V,F,A)$, identifying their diffusion matrix and
potential with $A$ and $V$.
The Bochner identity gives
\[
\Gamma_2(\varphi)(z)
\ge \nabla\varphi(z)^\top
\mathfrak R_z(V,F,A)\nabla\varphi(z),
\qquad \varphi\in C_c^\infty(X).
\]

\begin{proposition}[Ricci criterion for \Cref{ass:CDkappa}]
\label{thm:Ricci}
Under the common hypotheses above, assume that $P_t^\varepsilon$
and $\bar P_t$ belong to the preceding smooth diffusion setting. Set
\[
\mathfrak R^\eps(z):=\mathfrak R_z(V^\eps,F^\eps,A^\eps),
\qquad
\bar{\mathfrak R}(x):=\mathfrak R_x(\bar V,\bar F,\bar A).
\]
If, for some $\eps_0>0$ and $\kappa\ge0$,
\begin{equation}
\label{eq:Ricci-matrix-condition}
\mathfrak R^\eps+\kappa A^\eps\succeq0
\quad(0<\eps<\eps_0),
\qquad
\bar{\mathfrak R}+\kappa\bar A\succeq0,
\end{equation}
pointwise on the respective state spaces, then
$\mathcal L^\eps$ and $\bar{\mathcal L}$ satisfy
$CD(-\kappa,\infty)$, and \Cref{ass:CDkappa} holds.
\end{proposition}

\begin{proof}
The Bochner bound and the assumed matrix inequalities give
$CD(-\kappa,\infty)$ for both generators.
Applying \Cref{lem:curvature-to-gradient} to each semigroup
proves \Cref{ass:CDkappa} and the corresponding form-domain extensions.
\end{proof}

When $A$ is spatially constant, the Ricci matrix simplifies to
\[
\mathfrak R_z(V,F,A)
=A\bigl(\nabla^2V-\operatorname{Sym}(DF)\bigr)A.
\]
Consequently, when both diffusion matrices are spatially constant,
\Cref{eq:Ricci-matrix-condition} becomes
\[
\nabla^2V^\eps-\operatorname{Sym}(DF^\eps)
\succeq-\kappa(A^\eps)^{-1},
\qquad
\nabla^2\bar V-\operatorname{Sym}(D\bar F)
\succeq-\kappa\bar A^{-1}.
\]

\section{Case Studies: Averaging and Stiff-Potential Limits}
\label{sec:case_study}

In this section, we illustrate the verification of the coefficient, locking, dynamic and gradient assumptions in the averaging and stiff-potential settings.

%The displayed SDEs generate the semigroups used in the orbit convergence, while the nonstationary EPRs are evaluated for the corresponding invariant-measure adjoint diffusions.
%Time-dependent EPRs refer to the invariant-measure adjoint diffusions.
\subsection{Slow--fast averaging}

In the first example, we consider the averaging setting of
\cite{Khasminskii1963Averaging,Khasminskii1968ItoAveraging,
pavliotis2008multiscale}.

Let
\[
E=\R^{d_x}\times\R^{d_y},\qquad z=(x,y),\qquad
\Phi(x,y)=y\in\bar E:=\R^{d_y}.
\]

For each $\varepsilon>0$, let $Z_t^\varepsilon$ solve
\begin{equation}\label{eq:NL-avg-SDE}
  dZ_t^\varepsilon
  =b^\varepsilon(Z_t^\varepsilon)\,dt
   +\sqrt{2A^\varepsilon(Z_t^\varepsilon)}\,dW_t,
\end{equation}
where
\begin{equation}\label{eq:NL-avg-coeff}
  A^\varepsilon(z)
  =\begin{pmatrix}
    \varepsilon^{-1}a_1(z)&0\\[2pt]
    0&a_2(z)
  \end{pmatrix},\qquad b^\varepsilon(z)
  =\begin{bmatrix}\varepsilon^{-1}b_1(z)\\ b_2(z)\end{bmatrix}
  =-A^\varepsilon(z)\nabla V^\varepsilon(z)
   +\nabla\!\cdot A^\varepsilon(z)+\gamma^\varepsilon(z).
\end{equation}
Here $a_1$ and $a_2$ are smooth and locally uniformly elliptic,
and $a_1,a_2,b_1,b_2$ are independent of $\eps$.
The smooth potential $V^\eps$ and vector field
$\gamma^\eps=[\gamma_x^\eps;\gamma_y^\eps]$
give the divergence-form decomposition above.

Write $\bar\pi(dy)=\bar p(y)\,dy$, and 
assume additionally that $\bar p\in C^\infty(\bar E)$.
The frozen fast generator is
\begin{equation}\label{eq:fast-generator}
  \mathcal L_y^{\mathrm{fast}}\phi(x)
  :=\tr\!\big(a_1(x,y)\nabla_x^2\phi(x)\big)
    +b_1(x,y)\cdot\nabla_x\phi(x),
  \qquad \phi\in C_c^\infty(\R^{d_x}),
\end{equation}
while the slow generator is 
\[\mathcal L^{\mathrm{slow}}\psi
=\operatorname{tr}(a_2\nabla_y^2\psi)+b_2\cdot\nabla_y\psi.\]
For every $y\in\R^{d_y}$, $\mu^y$ has a strictly positive density
  and is the unique probability measure $\nu$ satisfying
  \[
    \int_{\R^{d_x}}\mathcal L_y^{\mathrm{fast}}\phi\,d\nu=0,
    \qquad \phi\in C_c^\infty(\R^{d_x}).
  \]
  The map $y\mapsto\mu^y$ is weakly continuous, and $a_2,b_2$ are
  $\mu^y$-integrable:
  \[
  \int_{\R^{d_x}}\bigl(\|a_2(x,y)\|+|b_2(x,y)|\bigr)\,\mu^y(dx)<\infty.
  \]
  
%Let $(P_t^\eps)_{t\ge0}$ be the backward semigroup of $Z^\eps$. The time-dependent EPR statements in this subsection concern the microscopic and macroscopic adjoint diffusions, with normalized relative densities $u^\eps(t)$ and $\bar u(t)$.
The averaged coefficients below generate a conservative diffusion satisfying the macroscopic assumptions of \Cref{sec:setting}. Define
\[
  (\mathcal P f)(y):=\int_{\R^{d_x}}f(x,y)\,\mu^y(dx),
  \qquad f\in C_b(E).
\]
The same notation is used for the $\mu^y$-integrable coefficient fields above.
By \Cref{lem:P-continuity}, $\mathcal P$ maps $C_b(E)$ into
$C_b(\bar E)$. For $\varphi\in C_c^\infty(\bar E)$, the block structure gives
\begin{equation}\label{eq:NL-avg-L-on-Phi}
  \mathcal L^\varepsilon(\varphi\circ\Phi)(x,y)
  =\tr\!\big(a_2(x,y)\nabla_y^2\varphi(y)\big)
   +b_2(x,y)\cdot\nabla_y\varphi(y).
\end{equation}
Hence
\begin{equation}\label{eq:NL-avg-barL}
  \bar{\mathcal L}\varphi
  :=\mathcal P\bigl(\mathcal L^\varepsilon(\varphi\circ\Phi)\bigr)
  =\bar A:\nabla_y^2\varphi+\bar b\cdot\nabla_y\varphi,
\end{equation}
where
\[
  \bar A(y)=\int a_2(x,y)\,\mu^y(dx),\qquad
  \bar b(y)=\int b_2(x,y)\,\mu^y(dx).
\]

\subsubsection{Linear Ornstein--Uhlenbeck model}
\label{subsubsec:OU-linear}
Let $d=d_x+d_y$ and
\[
  b^\varepsilon(z)=-I^\varepsilon Bz,\qquad
  A^\varepsilon=I^\varepsilon,\qquad
  I^\varepsilon=
  \begin{pmatrix}\varepsilon^{-1}I_{d_x}&0\\0&I_{d_y}\end{pmatrix},\qquad
  B=\begin{pmatrix}B_{11}&B_{12}\\B_{21}&B_{22}\end{pmatrix}.
\]
Thus
\begin{equation}\label{eq:SDE-OU}
  dZ_t^\varepsilon=-I^\varepsilon BZ_t^\varepsilon\,dt
  +\sqrt{2I^\varepsilon}\,dW_t.
\end{equation}
If $B_{11}$ is invertible, the effective slow process has
\begin{equation}\label{eq:averaged-slow1}
  \bar b(y)=-Cy,\qquad \bar A=I_{d_y},\qquad
  C:=B_{22}-B_{21}B_{11}^{-1}B_{12}.
\end{equation}

For the dynamical statements in the OU example, take
$\mathcal A=\mathcal M_{\mathrm{OU}}$, where
\[
\mathcal M_{\mathrm{OU}}
:=
\bigl\{
g^2:\ g\in C_b(E)\cap C_{\mathrm{pol}}^\infty(E),\ g>0
\bigr\},
\]
and $C_{\mathrm{pol}}^\infty(E)$ consists of smooth functions
whose derivatives of every order have at most polynomial growth.
Under the stability assumptions below,
\Cref{lem:OU-profile-compatibility} verifies the initial-profile
requirements for this choice:
$1\in\mathcal M_{\mathrm{OU}}\subset\mathcal M$ and
$\bar f:=\mathcal P f\in\bar{\mathcal M}$ for every
$f\in\mathcal M_{\mathrm{OU}}$.

\begin{example}\label{ex:OU}
Assume that $B_{11}$ and $C$ are positive stable, meaning that
all their eigenvalues have positive real parts.
For some $\varepsilon_0>0$, the family indexed by
$0<\varepsilon<\varepsilon_0$ has the following properties.
\begin{enumerate}
\item \Cref{ass:coeff-weak} holds.

\item \Cref{ass:locking} holds if and only if
\[
B_{11}=B_{11}^{\top},
\qquad B_{21}=B_{12}^{\top}.
\]

\item \Cref{ass:standing}  holds with
$\mathcal A=\mathcal M_{\mathrm{OU}}$, and
$(f,\mathcal Pf)$ is admissible for every $f\in\mathcal A$.

\item If $\Sym(B_{11})\succ0$, then \Cref{ass:CDkappa} holds.
\end{enumerate}
The proofs are given in \Cref{subsubsec:OU}.
\end{example}

For every $f\in\mathcal M_{\mathrm{OU}}$ with $\bar f=\mathcal Pf$,
\Cref{thm:FI-conv} gives the housekeeping lower semicontinuity at every
positive time and, if $\Sym(B_{11})\succ0$, excess convergence.
Under the locking identities in Item~2, positive stability gives
$B_{11}\succ0$; hence both EPR components, and their sum, converge
for every such initial pair at every positive time.

\subsubsection{Nonlinear averaging with a common invariant measure}
\label{subsubsec:nonlinear-fastslow}
We consider a common invariant measure $\pi^\eps=\pi$.
Variation of the slow force along fast fibres can still produce
housekeeping loss.

In \Cref{eq:NL-avg-coeff}, set
\begin{equation}\label{eq:tractable-drift}
V^\eps\equiv V,\qquad \gamma^\eps\equiv\gamma=[0;\gamma_y].
\end{equation}
Writing $\pi(dz):=Z^{-1}e^{-V(z)}\,dz$, assume
\begin{equation}
    \label{eq:tractable}
    \int_E\!\left[\tr(a_2)+|b_2|+
\gamma_y^\top a_2^{-1}\gamma_y\right]d\pi<\infty.
\end{equation}
Then $\pi^\eps=\pi=\mu^y(dx)\bar\pi(dy)$ for every $\eps>0$ (\Cref{lem:tractable-static}).

\begin{proposition}[Coefficient verification]
\label{prop:tractable-assumptions}
Under the common-invariant-measure setting \Cref{eq:tractable-drift} and the integrability condition \eqref{eq:tractable},
\Cref{ass:coeff-weak} holds.
\end{proposition}

\begin{proof}
By \Cref{lem:tractable-static},
$\pi(dx,dy)=\mu^y(dx)\bar\pi(dy)$.
Write $\pi(dx,dy)=p(x,y)\,dx\,dy$.
The integrability condition and Cauchy--Schwarz inequality give
$\gamma_y\in L^1(\pi)$.
Integrating
$\gamma_y p=b_2p-\nabla_y\cdot(a_2p)$
in $x$, using compact cutoffs, yields
\begin{equation}\label{eq:gammay-calculate}
(\mathcal P\gamma_y)\bar p
=
\bar b\,\bar p-\nabla_y\cdot(\bar A\bar p)
=
\bar\gamma\,\bar p
\end{equation}
in distributions on $\bar E$.
Hence $\mathcal P\gamma_y=\bar\gamma$ $\bar\pi$-a.e., and
\[
J^\eps=\bar\gamma\bar\pi,\qquad
Q^\eps=a_2\pi=:Q,\qquad
\Phi_\#Q=\bar A\bar\pi.
\]
Moreover,
\[
\mathcal J_{\mathrm{hk,proj}}^\eps
=
\int_E\gamma_y^\top a_2^{-1}\gamma_y\,d\pi.
\]
These measures and the last integral are finite by the integrability
condition and independent of $\eps$.
This verifies all three clauses of \Cref{ass:coeff-weak}.
\end{proof}
\begin{proposition}[Stationary loss and locking]
\label{prop:locking-tractable}
Set
\begin{equation}
\label{eq:Fyandv}
F_y:=a_2^{-1}\gamma_y,\qquad
v(y):=\int_{\mathbb R^{d_x}}
|F_y(x,y)-\bar F(y)|_{a_2(x,y)}^2\,\mu^y(dx).
\end{equation}
Then $v\in L^1(\bar\pi)$, $\bar\sigma_{\mathrm{ss}}<\infty$, and
\[
\sigma_{\mathrm{ss}}^\eps
=\bar\sigma_{\mathrm{ss}}+\mathcal V_{\mathrm{ss}}
\ge\bar\sigma_{\mathrm{ss}},
\qquad
\mathcal V_{\mathrm{ss}}:=\int_{\bar E}v\,d\bar\pi.
\]
In particular, stationary lower semicontinuity holds, and
$\sigma_{\mathrm{ss}}^\eps\to\bar\sigma_{\mathrm{ss}}$
as $\eps\downarrow0$ if and only if
\[
F_y=\bar F\circ\Phi
\qquad\pi\text{-a.e.}
\]
This condition is equivalent to \Cref{ass:locking}.
\end{proposition}

\begin{proof}
By definition, $\bar A=\mathcal Pa_2$, while
\Cref{eq:gammay-calculate} gives
$\bar\gamma=\mathcal P\gamma_y$ $\bar\pi$-a.e.
Thus the macroscopic force satisfies
\[
\bar F=(\mathcal Pa_2)^{-1}\mathcal P(a_2F_y)
\qquad\bar\pi\text{-a.e.}
\]
Expanding the square in \Cref{eq:Fyandv} yields
\[
\mathcal P(F_y^\top a_2F_y)
=\bar F^\top\bar A\bar F+v
\qquad\bar\pi\text{-a.e.}
\]
The integrability assumption and the block structure give
$\sigma_{\mathrm{ss}}^\eps
=\int_E F_y^\top a_2F_y\,d\pi<\infty$.
Integration therefore proves the stationary loss formula and
the asserted finiteness.

By disintegration,
\[
\mathcal V_{\mathrm{ss}}
=\int_E|F_y-\bar F\circ\Phi|_{a_2}^2\,d\pi.
\]
Since $a_2$ is positive definite, this loss vanishes exactly
when $F_y=\bar F\circ\Phi$ $\pi$-a.e.
Finally, \Cref{prop:tractable-assumptions} verifies
\Cref{ass:coeff-weak}, so \Cref{thm:ss-lsc-locking}
identifies the stationary convergence with \Cref{ass:locking}.
\end{proof}

We now turn to the dynamical results and set
\[
\mathcal A=\mathcal M_{\mathrm{av}}
:=\bigl\{f\in\mathcal M:
\mathcal Pf\in\bar{\mathcal M}\bigr\},
\qquad
\bar f:=\mathcal Pf.
\]
\Cref{lem:tractable-static} gives the invariant-measure disintegration and
common normalization, verifying \Cref{ass:standing}(i) with $\Pi=\pi$.
Under the coefficient hypotheses of
\citet[Theorem~4.4]{BardiKouhkouh2023},
semigroup convergence holds locally uniformly at every positive time
for compactly supported smooth profiles. Their condition~(C1) requires
a constant fast-noise matrix whose covariance is a positive multiple
of the identity. The cutoff argument in \Cref{subsubsec:si-gaussian-multiplicative}
extends this convergence to $\mathcal M_{\mathrm{av}}$; common
normalization then yields \Cref{eq:dyn-conv-standing}, verifying
\Cref{ass:standing}(ii). The averaging hypotheses also imply the
growth bounds required by \Cref{lem:avg-cutoffs}, which provides
the cutoffs in \Cref{ass:standing}(iii) for both systems.

Assume additionally the hypotheses of one of the microscopic gradient
criteria in \Cref{subsubsec:Ito-Kunita}, such as \Cref{thm:CD-negative-short}.
The chosen criterion gives the microscopic clause of
\Cref{ass:CDkappa}. Together with the preceding orbit and cutoff
conditions, \Cref{lem:avg-gradient-inheritance} then yields the macroscopic clause
with the same constant.

Under these joint hypotheses,
\Cref{prop:tractable-assumptions,prop:locking-tractable} and \Cref{thm:FI-conv} give,
for every $f\in\mathcal M_{\mathrm{av}}$ with $\bar f=\mathcal P f$
and every $t>0$,
\[
\sigma_{\mathrm{ex}}^\eps(t)
\longrightarrow \bar\sigma_{\mathrm{ex}}(t)<\infty,
\qquad
\liminf_{\eps\downarrow0}\sigma_{\mathrm{hk}}^\eps(t)
\ge \bar\sigma_{\mathrm{hk}}(t).
\]
The housekeeping EPR converges to its macroscopic counterpart
for all such pairs and times if and only if
$F_y=\bar F\circ\Phi$ $\pi$-a.e.
Since the excess EPR converges, the total EPR satisfies the
same lower semicontinuity and preservation criterion.

The following multiplicative-noise example satisfies the preceding
averaging and gradient-estimate hypotheses but not locking.

Let $E=\mathbb R\times\mathbb R^2$, $\bar E=\mathbb R^2$
and $\Phi(x,y)=y$, and let $\nu_d$ denote the standard
Gaussian law on $\mathbb R^d$. Choose a nonzero
$\vartheta_\star\in C_c^\infty((0,\infty))$ with
$0\leq\vartheta_\star\leq1$, extend it by zero to $\mathbb R$, and set
\[
\begin{gathered}
d_\star(y)=1+\vartheta_\star(|y|^2),\qquad
c_\star(y)=\vartheta_\star(|y|^2)\begin{pmatrix}-y_2\\y_1\end{pmatrix},\qquad
r_\star(x)=1+\sin x.
\end{gathered}
\]
Consider the averaging model with
\[
a_1=1,\qquad a_2(y)=d_\star(y)I_2,\qquad
V^\varepsilon(x,y)=\tfrac12(x^2+|y|^2),\qquad
\gamma^\varepsilon(x,y)=\begin{pmatrix}0\\r_\star(x)c_\star(y)\end{pmatrix}.
\]
Its drifts are
\[
b_1(x,y)=-x,\qquad
b_2(x,y)=\nabla d_\star(y)-d_\star(y)y+r_\star(x)c_\star(y).
\]
The fast noise coefficient is constant, whereas the slow noise
coefficient is nonconstant.

\begin{example}[A common Gaussian invariant measure]
\label{ex:gaussian-multiplicative}
For every $\varepsilon>0$, the invariant measure and frozen fast measure are
\[
\pi^\varepsilon=\pi=\nu_1\otimes\nu_2,\qquad
\mu_y=\nu_1,\qquad \bar\pi=\nu_2.
\]
The averaged coefficients are
\[
\bar A(y)=d_\star(y)I_2,\qquad
\bar b(y)=\nabla d_\star(y)-d_\star(y)y+c_\star(y),\qquad
\bar\gamma(y)=c_\star(y).
\]
Both the microscopic and macroscopic diffusions are nonreversible.
For $0<\varepsilon\leq1/4$, the following properties hold.
\begin{enumerate}
\item \Cref{ass:coeff-weak} holds.
\Cref{ass:standing} holds with microscopic profile class
$\mathcal M_{\mathrm{av}}$.
For every $f\in\mathcal M_{\mathrm{av}}$, the pair
$(f,\mathcal P f)$ is admissible.
\item \Cref{ass:CDkappa} holds with a constant $\kappa$
independent of $\varepsilon$.
\item Locking fails. More precisely,
\[
\bar\sigma_{\mathrm{ss}}
=\int_{\mathbb R^2}\frac{|c_\star(y)|^2}{d_\star(y)}\,d\nu_2(y)>0,
\qquad
\sigma_{\mathrm{ss}}^\varepsilon
=\frac{3-e^{-2}}{2}\,\bar\sigma_{\mathrm{ss}}.
\]
\end{enumerate}
\end{example}

The verification is given in \Cref{subsubsec:si-gaussian-multiplicative}.
For every $f\in\mathcal M_{\mathrm{av}}$, with $\bar f=\mathcal P f$,
the normalized trajectories satisfy, for every $t>0$,
\[
\sigma_{\mathrm{ex}}^\varepsilon(t)\longrightarrow
\bar\sigma_{\mathrm{ex}}(t),\qquad
\sigma_{\mathrm{hk}}^\varepsilon(t)\longrightarrow
\frac{3-e^{-2}}{2}\,\bar\sigma_{\mathrm{hk}}(t)
>\bar\sigma_{\mathrm{hk}}(t),
\]
as $\varepsilon\downarrow0$. 

We return to the general common-invariant-measure setting.
Since $\pi$, $a_1$ and $a_2$ are independent of $\eps$,
the gradients can be compared in fixed weighted Hilbert spaces.
The block structure gives
\[
\frac14\sigma_{\mathrm{ex}}^\eps(t)
=
\eps^{-1}
\bigl\|\nabla_x\sqrt{u^\eps(t)}\bigr\|_{L^2(\pi;a_1)}^2
+
\bigl\|\nabla_y\sqrt{u^\eps(t)}\bigr\|_{L^2(\pi;a_2)}^2,
\]
where
\[
\|w\|_{L^2(\pi;a_i)}^2
:=
\int_E w^\top a_iw\,d\pi,
\qquad i=1,2.
\]
At a fixed positive time, orbit convergence and convergence
of the excess EPR to its macroscopic value yield the following
gradient convergence result, without locking.
\begin{theorem}[Weighted $L^2$ gradient convergence]
\label{thm:L2-strong}
In the setting of \Cref{subsubsec:nonlinear-fastslow},
fix $f\in\mathcal M_{\mathrm{av}}$, set $\bar f=\mathcal Pf$,
and let $t>0$.
Suppose that \Cref{eq:dyn-conv-standing} holds for this pair at time $t$ and that
\[
\sigma^\eps_{\mathrm{ex}}(t)
\longrightarrow\bar\sigma_{\mathrm{ex}}(t)<\infty.
\]
Then
\begin{align*}
\textnormal{(i)}\quad&
\|\nabla_x\sqrt{u^\eps(t)}\|_{L^2(\pi;a_1)}^2=o(\eps),
&
\textnormal{(ii)}\quad&
\nabla_y\sqrt{u^\eps(t)}\to\nabla_y\sqrt{\bar u(t)\circ\Phi}
\quad\text{in }L^2(\pi;a_2),\\
\textnormal{(iii)}\quad&
\|\nabla_xu^\eps(t)\|_{L^2(\pi;a_1)}^2=o(\eps),
&
\textnormal{(iv)}\quad&
\nabla_yu^\eps(t)\to\nabla_y(\bar u(t)\circ\Phi)
\quad\text{in }L^2(\pi;a_2).
\end{align*}
\end{theorem}
%See \Cref{subsubsec:L2-str} for the proof.
\begin{proof}%[Proof of \Cref{thm:L2-strong}]
Fix $t>0$.
By \Cref{lem:tractable-static}, $m^\eps=\bar m=m$.
Write $\|\cdot\|_p=\|\cdot\|_{L^p(\pi)}$ and set
\[
M_f:=\frac{\|f\|_\infty}{m},
\qquad
v^\eps:=\sqrt{u^\eps(t)},
\qquad
v:=\sqrt{\bar u(t)}\circ\Phi.
\]
The Markov property and orbit convergence give
\[
0\le v^\eps,v\le M_f^{1/2},\qquad
\|v^\eps-v\|_2^2
\le \|u^\eps(t)-\bar u(t)\circ\Phi\|_1
\longrightarrow0,
\]
where the last limit follows by dominated convergence. Put
\[
h^\eps=a_2^{1/2}\nabla_yv^\eps,\qquad
\alpha^\eps=\eps^{-1}
\|a_1^{1/2}\nabla_xv^\eps\|_2^2,\qquad
\beta^\eps=\|h^\eps\|_2^2.
\]
By the assumed excess EPR convergence,
\begin{equation}\label{eq:L2-FI-sum}
\alpha^\eps+\beta^\eps
=\tfrac14\sigma_{\mathrm{ex}}^\eps(t)
\longrightarrow\tfrac14\bar\sigma_{\mathrm{ex}}(t)<\infty.
\end{equation}

Thus $(h^\eps)$ is bounded in $L^2(\pi)$.
Since $a_2$ is locally uniformly elliptic and $\pi$ has a positive
continuous density, passing to the limit in
$\nabla_yv^\eps=a_2^{-1/2}h^\eps$ in distributions identifies every
weak subsequential limit as $h:=a_2^{1/2}\nabla_yv$.
Consequently,
\(
h^\eps\rightharpoonup h
\,\text{in }L^2(\pi).
\)

Disintegration and $\bar A=\mathcal Pa_2$ yield
\[
\|h\|_2^2
=\int_{\bar E}
\bigl|\nabla_y\sqrt{\bar u(t,y)}\bigr|_{\bar A(y)}^2
\,d\bar\pi(y)
=\tfrac14\bar\sigma_{\mathrm{ex}}(t).
\]
Weak lower semicontinuity and \eqref{eq:L2-FI-sum} therefore imply
\[
\alpha^\eps\longrightarrow0,\,
\beta^\eps\longrightarrow\|h\|_2^2.
\]

Weak convergence together with convergence of the norms gives
$h^\eps\to h$ strongly in $L^2(\pi)$, proving
\textup{(i)}--\textup{(ii)}.

The chain rule gives
\[
\|a_1^{1/2}\nabla_xu^\eps(t)\|_2^2
\le4M_f\eps\alpha^\eps=o(\eps),
\]
which proves \textup{(iii)}. Finally,
\[
\tfrac12 a_2^{1/2}
\nabla_y\bigl(u^\eps(t)-\bar u(t)\circ\Phi\bigr)
=v^\eps(h^\eps-h)+(v^\eps-v)h.
\]
The first term tends to zero in $L^2(\pi)$ by boundedness of
$v^\eps$ and strong convergence of $h^\eps$.
The second does so by dominated convergence, since orbit convergence
gives $v^\eps\to v$ pointwise and
$|(v^\eps-v)h|^2\le M_f|h|^2\in L^1(\pi)$.
This proves \textup{(iv)}.
\end{proof}
%The weighted gradient conclusions of \Cref{thm:L2-strong} also hold.

%Reusable sufficient conditions for the microscopic gradient estimate in \Cref{ass:CDkappa}, with proofs for spatially constant diffusion and multiplicative noise in slow--fast systems, are given in \Cref{subsubsec:Ito-Kunita}.

%Then we introduce a nonlinear model with multiplicative noise.

\subsection{A reversible stiff-potential limit}
\label{subsec:stiff}

In the second example, we consider a stiff potential that confines a diffusion to an affine
manifold as $\eps\downarrow0$.

Let $E=\R^{d_x}\times\R^{d_y}$, $\bar E=\R^{d_y}$, and fix
$B\in\mathbb S_{++}^{d_x}$, $H\in\R^{d_x\times d_y}$ and
$b\in\R^{d_x}$. For $V\in C^2(E)$ with $\nabla V$ of at most
linear growth, consider
\begin{equation}\label{eq:stiff-nl-SDE}
dZ_t^\eps
=-\nabla(V+\eps^{-2}U)(Z_t^\eps)\,dt+\sqrt{2}\,dW_t,
\qquad
U(x,y):=\frac12(x-Hy-b)^\top B(x-Hy-b).
\end{equation}
Set
\[
\tilde H:=I_{d_y}+H^\top H,\qquad
\hat H:=I_{d_x}+HH^\top,\qquad
\iota(\omega):=(H\omega+b,\omega),
\]
and define
\[
\Phi(x,y):=\tilde H^{-1}\bigl(y+H^\top(x-b)\bigr).
\]
Then $\Phi\circ\iota=\Id_{\bar E}$, and $\iota\circ\Phi$ is
the orthogonal projection onto the constraint manifold
$M:=\iota(\bar E)=\{U=0\}$.
Since $D\Phi\nabla U=0$, the stiff drift preserves $\Phi$.

Introduce the global coordinates
\begin{equation}\label{eq:stiff-tubular}
\mathcal T(\omega,r):=
\bigl(H\omega+b+\hat H^{-1}r,\,
      \omega-H^\top \hat H^{-1}r\bigr).
\end{equation}
Here $\omega=\Phi(x,y)$ parametrizes position along $M$, while
$r=x-Hy-b$ measures transverse displacement; in particular,
$U(\mathcal T(\omega,r))=\frac12r^\top Br$.

Assume that there exist a measurable function $\Psi:\bar E\to\R$,
constants $C_0,C_1\ge0$, and $\eps_*>0$ such that
\begin{equation}\label{eq:stiff-geom}
e^{-\Psi}\in L^1(\bar E),\qquad
V(\mathcal T(\omega,r))
\ge\Psi(\omega)-C_0-C_1|r|^2,\qquad
C_1\eps_*^2<\frac12\lambda_{\min}(B),
\end{equation}
for all $(\omega,r)\in\bar E\times\R^{d_x}$.
The stiff quadratic term controls transverse displacement, while
the integrability of $e^{-\Psi}$ provides confinement along $M$.
Under these assumptions,
$\operatorname{dist}(Z_t^\eps,M)=O(\eps)$ in $L^2$ for every
fixed $t>0$, uniformly over initial points in compact sets;
see \Cref{eq:stiff-fast-est}.

Put $\bar V:=V\circ\iota$. Condition \eqref{eq:stiff-geom}
ensures that the Gibbs laws
\[
\pi^\eps(dz)
:=Z_\eps^{-1}e^{-V(z)-\eps^{-2}U(z)}\,dz,
\qquad
\bar\pi(d\omega)
:=\bar Z^{-1}e^{-\bar V(\omega)}\,d\omega
\]
are well defined for $0<\eps\le\eps_*$, where $Z_\eps$ and
$\bar Z$ are the normalizing constants.
The limiting semigroup $\bar P_t$ has generator
\begin{equation}\label{eq:stiff-limit-generator}
\bar{\mathcal L}g
=\tr(\bar A\nabla^2g)
-\langle\bar A\nabla\bar V,\nabla g\rangle,
\qquad g\in C_c^\infty(\bar E),
\end{equation}
with
$\bar A:=\tilde H^{-1}=D\Phi D\Phi^\top$.
Both generators are of the reversible-divergence form, with
$\gamma^\eps=\bar\gamma=0$.
Thus their stationary and housekeeping EPRs vanish, and the
total EPR equals the excess EPR.

Choose $0<\eps_0\le\eps_*$ as the parameter upper bound
in the definition of $\mathcal M$. For the dynamical statements,
use the profile class
\[
\mathcal A=\mathcal M^{\mathrm{stiff}}
:=\{f\in\mathcal M:f\circ\iota\in\bar{\mathcal M}\}.
\]
This class contains $1$, and the additional restriction requires
$\sqrt{f\circ\iota}$ to have finite macroscopic form energy.
For $f\in\mathcal A$, set $\bar f:=f\circ\iota$ and use the
corresponding normalized initial data.

\begin{example}[Reversible stiff-potential limit]
\label{ex:stiff}
Under the preceding assumptions, the following hold.
\begin{enumerate}
\item
\Cref{ass:standing} holds with $\Pi=\iota_\#\bar\pi$, and
$(f,f\circ\iota)$ is admissible for every $f\in\mathcal A$.

\item
Suppose additionally that, for some $\kappa\ge0$,
\begin{equation}\label{eq:stiff-Hessian}
\nabla^2(V+\eps^{-2}U)\succeq-\kappa I_{d_x+d_y},
\qquad 0<\eps<\eps_0.
\end{equation}
Then \Cref{ass:CDkappa} holds for both semigroups with the
same $\kappa$. Consequently, for every $f\in\mathcal A$ and
$0<\tau<T<\infty$,
\[
\begin{aligned}
\sup_{t\in[\tau,T]}
|\sigma^\eps(t)-\bar\sigma(t)|
&=
\sup_{t\in[\tau,T]}
|\sigma_{\mathrm{ex}}^\eps(t)-\bar\sigma_{\mathrm{ex}}(t)|
\longrightarrow0
\qquad(\eps\downarrow0).
\end{aligned}
\]
\end{enumerate}
\end{example}

Since $\nabla^2U\succeq0$, the simpler bound
$\nabla^2V\succeq-\kappa I_{d_x+d_y}$ is sufficient for
\Cref{eq:stiff-Hessian}.
Restricting \Cref{eq:stiff-Hessian} to tangent directions gives
$\nabla^2\bar V\succeq-\kappa\tilde H=-\kappa\bar A^{-1}$,
the corresponding macroscopic curvature bound.
The proof is given in \Cref{subsec:stiff-proof}.

\section{Discussion and perspectives}\label{sec:discussion}
%The averaging examples make the EPR preservation criteria explicit and quantify the housekeeping dissipation lost in the coarse-graining limits considered here.

Our proof of excess EPR convergence combines the relative-entropy
balance with the monotonicity of
$t\mapsto e^{-2\kappa t}\sigma_{\mathrm{ex}}^\eps(t)$,
obtained from \Cref{ass:CDkappa}.
This controls the instantaneous rate through averaging over
neighbouring time intervals.
A possible extension is to obtain comparable time control
under weaker assumptions, using a different auxiliary
functional while still identifying the limit of the
original excess EPR.

A second direction is to study dissipation rates associated
with other convex entropy functionals.
When such rates are squared norms of transformed gradients
in a fixed weighted $L^2$ space, the argument of
\Cref{thm:L2-strong} yields strong gradient convergence
if the gradients converge weakly and the dissipation rates
converge to the macroscopic dissipation, identified as the
squared norm of that weak limit.

Beyond the examples in \Cref{sec:case_study}, it would be
useful to develop unified coefficient conditions and
relaxation estimates for unresolved variables that imply
the orbit convergence in \Cref{ass:standing}.
The initial layer also warrants further studies.
For a fixed admissible reference profile, the initial
excess EPR is finite for each $\eps$ but needs not to remain
bounded as $\eps\downarrow0$.
Estimates on $\eps$-dependent short timescales could
quantify this early dissipation and its transition
to the positive-time limit.

Finally, a pathwise extension is to study the random total
entropy production accumulated over $[\tau,T]$, with
$0<\tau<T$, for the invariant-measure adjoint diffusions
used here.
Assuming mutual absolute continuity, this is the log-likelihood
ratio of the forward path law to the law obtained by starting
the same dynamics from the forward final distribution and
the resulting reversal paths.
Under their second-order averaging and compatibility
assumptions, \citet{GeJiaJin2021JSP} establish joint
weak limits of the slow process and the accumulated entropy
production, allowing an anomalous contribution beyond
the effective entropy production.

For families in the present framework outside those
assumptions, a natural question is which additional conditions
yield analogous limits, how any defect relative to the
effective entropy production can be characterized, and
when it vanishes.
Suitable uniform-integrability estimates could then help to
relate the mean defect to the time-integrated housekeeping
loss.

\section*{Acknowledgment}
This work was supported by the National Natural Science Foundation of China (12671641) and the Guangdong Provincial Key Laboratory of Mathematical and Neural Dynamical Systems (2024B1212010004).

\section*{Conflict of Interest Statement}
The authors declare that they have no known competing financial interests or personal relationships that could have appeared to influence the work reported in this paper.
\section*{Data Availability Statement}
No new data were created or analyzed in this study. Data sharing is not applicable to this article.

\bibliographystyle{imsart-nameyear}
\bibliography{references}

\begin{suppinfo}
\section{Proofs for Results in Section 3}
\label{sec:mr-app}
\subsection{Stationary entropy production}
\label{subsec:static-proof}
\begin{lemma}[Stationary housekeeping dissipation and stationary recovery]
\label{lem:hk_ss_finite}
Suppose \Cref{ass:coeff-weak} holds. Set
$A_\Phi^\eps:=D\Phi A^\eps D\Phi^\top$ and, for
$\psi\in C_b(\bar E;\mathbb R^n)$, define
\begin{align*}
\widetilde\sigma^\eps(\psi)
&:=\int_E
\bigl[
2\langle\psi\circ\Phi,D\Phi\gamma^\eps\rangle
-|\psi\circ\Phi|_{A_\Phi^\eps}^2
\bigr]\,d\pi^\eps,\qquad
\widetilde\sigma(\psi)
:=\int_{\bar E}
\bigl[
2\langle\psi,\bar\gamma\rangle-|\psi|_{\bar A}^2
\bigr]\,d\bar\pi.
\end{align*}
Then $\widetilde\sigma^\eps(\psi)\to\widetilde\sigma(\psi)$
for every such $\psi$, and
\begin{equation}\label{eq:macro-hk-duality}
\bar\sigma_{\mathrm{ss}}
=
\sup_{\psi\in C_b(\bar E;\mathbb R^n)}
\widetilde\sigma(\psi)
\le
\sup_{0<\eps<\eps_0}\mathcal J_{\mathrm{hk,proj}}^\eps
<\infty.
\end{equation}
Moreover,  the fields
\[
\psi_k:=\frac{k\bar F}{\max\{k,|\bar F|\}},
\qquad k\ge1,
\]
belong to $C_b(\bar E;\mathbb R^n)$ and satisfy
\begin{equation}\label{eq:recovery_L2_weighted}
\int_{\bar E}|\psi_k-\bar F|_{\bar A}^2\,d\bar\pi
\longrightarrow0,
\qquad
\|\psi_k\|_\infty\le k.
\end{equation}
Thus $(\psi_k)$ is a recovery sequence in the
sense of \Cref{def:recovery}.
\end{lemma}

\begin{proof}
The integrals are finite, and their convergence follows from
the projected-current and projected-diffusivity convergences
in \Cref{ass:coeff-weak}. Completing the square gives
\[
\widetilde\sigma^\eps(\psi)
\le \mathcal J_{\mathrm{hk,proj}}^\eps,
\qquad
\widetilde\sigma(\psi)
\le C:=
\sup_{0<\eps<\eps_0}\mathcal J_{\mathrm{hk,proj}}^\eps
<\infty.
\]
The standing coefficient regularity makes $\bar F$ continuous,
so $\psi_k\in C_b(\bar E;\mathbb R^n)$ and
$\|\psi_k\|_\infty\le k$. Writing
$\psi_k=\alpha_k\bar F$, where
$\alpha_k:=k/\max\{k,|\bar F|\}\uparrow1$, monotone convergence
yields
\[
\widetilde\sigma(\psi_k)
=
\int_{\bar E}(2\alpha_k-\alpha_k^2)
|\bar F|_{\bar A}^2\,d\bar\pi
\uparrow\bar\sigma_{\mathrm{ss}}\le C.
\]
Since $\bar\sigma_{\mathrm{ss}}<\infty$, completing the square
now gives, for every $\psi\in C_b(\bar E;\mathbb R^n)$,
\[
\widetilde\sigma(\psi)
=
\bar\sigma_{\mathrm{ss}}
-\int_{\bar E}|\psi-\bar F|_{\bar A}^2\,d\bar\pi
\le\bar\sigma_{\mathrm{ss}}.
\]
Together with the truncation limit, this proves
\Cref{eq:macro-hk-duality}; taking $\psi=\psi_k$ in the same
identity proves \Cref{eq:recovery_L2_weighted}.
\end{proof}
\subsection{Nonstationary entropy production}
\label{subsec:dynamic-proof}
Throughout this appendix, $(P_t,\pi,\Gamma)$ denotes either system,
and $u_t=P_tf_0$ has the admissible normalized initial data of
\Cref{sec:setting}. We use the relative entropy \(
    \F(t)\) defined in \Cref{eq:F-def}.

\begin{lemma}[Integrated de Bruijn identity]
\label{lem:de-bruijn}
Suppose the cutoff condition in \Cref{ass:standing}(iii) holds.
Then, for $0<s<t$,
\begin{equation}\label{eq:integrated-de-bruijn}
\F(t)+\int_s^t \sigma_{\mathrm{ex}}(r)\,dr=\F(s).
\end{equation}
Consequently, $\F\in \mathrm{AC}_{\mathrm{loc}}((0,\infty))$
and $\F'=-\sigma_{\mathrm{ex}}$ almost everywhere.
If $\sigma_{\mathrm{ex}}$ is continuous on $(0,\infty)$,
then $\F\in C^1((0,\infty))$ and the derivative identity
holds everywhere.
\end{lemma}

\begin{proof}
Let $\chi_R$ be the prescribed cutoffs and put
$M:=\|f_0\|_\infty$, so that $0\le u_r\le M$.
Local parabolic regularity justifies the following computations
with compactly supported test functions.
Applying invariance to $\chi_R^2u_r^2/2$, and using the chain
and product rules followed by Young's inequality, gives
\[
\frac12\int_s^t\!\int_X
\chi_R^2\Gamma(u_r)\,d\pi\,dr
\le \frac{M^2}{2}
+CM^2(t-s)\int_X
\bigl(|\mathcal L\chi_R|+\Gamma(\chi_R)\bigr)\,d\pi.
\]
Letting $R\to\infty$ therefore yields
$\int_s^t\int_X\Gamma(u_r)\,d\pi\,dr<\infty$.

For $\delta>0$, set
\[
\varphi_\delta(a)
:=(a+\delta)\log(a+\delta)-\delta\log\delta.
\]
Applying invariance to $\chi_R\varphi_\delta(u_r)$ and using
the chain and product rules yields
\[
\int_X\chi_R\varphi_\delta(u_t)\,d\pi
+\int_s^t\!\int_X
\chi_R\frac{\Gamma(u_r)}{u_r+\delta}\,d\pi\,dr
=
\int_X\chi_R\varphi_\delta(u_s)\,d\pi
+\mathcal R_{R,\delta},
\]
where
\[
\mathcal R_{R,\delta}
=
-\int_s^t\!\int_X
\left[
\varphi_\delta(u_r)\mathcal L\chi_R
+2\varphi_\delta'(u_r)\Gamma(\chi_R,u_r)
\right]\,d\pi\,dr.
\]
Since $\varphi_\delta$ and $\varphi_\delta'$ are bounded on
$[0,M]$, Cauchy--Schwarz and the preceding energy bound give
\[
|\mathcal R_{R,\delta}|
\le C_{\delta,M,s,t}
\left[
\int_X|\mathcal L\chi_R|\,d\pi
+
\left(\int_X\Gamma(\chi_R)\,d\pi\right)^{1/2}
\right]
\longrightarrow0.
\]
Let $R\to\infty$ and then $\delta\downarrow0$.
Uniform convergence of $\varphi_\delta$ to $a\mapsto a\log a$
on $[0,M]$, monotone convergence, and
$\Gamma(u_r)/u_r=4\Gamma(\sqrt{u_r})$
prove \Cref{eq:integrated-de-bruijn}.
The remaining assertions follow from the fundamental theorem
of calculus.
\end{proof}

\begin{lemma}[Convergence of the relative entropies]
\label{lem:free-energy-conv}
Suppose \Cref{ass:standing} holds. Then, for every $t>0$,
\begin{equation}\label{eq:free-energy-pointwise}
\F^\eps(t)\longrightarrow\bar{\F}(t).
\end{equation}
\end{lemma}

\begin{proof}
Since $f\in C_b(E)$ is strictly positive,
\[
m^\eps:=\int_Ef\,d\pi^\eps
\longrightarrow\int_Ef\,d\Pi>0.
\]
Normalization and Markov contractivity therefore give
$0\le u^\eps(t),\bar u(t)\le M$ for some $M<\infty$,
uniformly in $t\ge0$ and sufficiently small $\eps$.
For fixed $t>0$, local uniform orbit convergence and tightness give
\begin{equation}
\label{eq:L1-alter}
\int_E|u^\eps(t)-\bar u(t)\circ\Phi|\,d\pi^\eps
\longrightarrow0.
\end{equation}
The function $\varphi(a):=a\log a$, with $\varphi(0):=0$,
is uniformly continuous on $[0,M]$. Hence
\[
\int_E
|\varphi(u^\eps(t))-\varphi(\bar u(t)\circ\Phi)|
\,d\pi^\eps\longrightarrow0.
\]
Finally, $\pi^\eps\rightharpoonup\Pi$ and
$\Phi_\#\Pi=\bar\pi$ identify the limit of the second integrand
as $\bar{\F}(t)$.
\end{proof}
\begin{lemma}[Extension of gradient estimates]
\label{lem:gradient-form-extension}
Under the standing semigroup and form assumptions, let $c\in\mathbb R$, and suppose that one of the estimates
\[
\sqrt{\Gamma(P_t\varphi)}
\le e^{ct}P_t\sqrt{\Gamma(\varphi)},
\qquad
\Gamma(P_t\varphi)
\le e^{2ct}P_t\Gamma(\varphi)
\]
holds for every $\varphi\in C_c^\infty(X)$ and $t\ge0$.
Then the same estimate holds $\pi$-almost everywhere for every
$f\in \mathcal D(\mathcal E)$ and $t\ge0$. Moreover,
\[
P_tf\in \mathcal D(\mathcal E),
\qquad
\mathcal E(P_tf)\le e^{2ct}\mathcal E(f).
\]
\end{lemma}
\begin{proof}
Choose $f_n\in C_c^\infty(X)$ converging to $f$ in the form norm.
The operator-core assumption gives
$\operatorname{Dom}(\mathcal L)\subset \mathcal D(\mathcal E)$, so
$P_tf_n\in \mathcal D(\mathcal E)$.
Applying either estimate to $f_n-f_m$ and integrating, using
invariance and Jensen's inequality when needed, gives
\[
\mathcal E(P_tf_n-P_tf_m)
\le e^{2ct}\mathcal E(f_n-f_m).
\]
Closedness and $L^2(\pi)$ contractivity therefore imply
$P_tf_n\to P_tf$ in the form norm and give the asserted energy bound.
For the first estimate, use
\[
\bigl\|\sqrt{\Gamma(g)}-\sqrt{\Gamma(h)}\bigr\|_{L^2(\pi)}
\le \mathcal E(g-h)^{1/2}
\]
and $L^2(\pi)$ contractivity to pass to the limit.
For the second, form convergence gives
$\Gamma(f_n)\to\Gamma(f)$ and
$\Gamma(P_tf_n)\to\Gamma(P_tf)$ in $L^1(\pi)$;
$L^1(\pi)$ contractivity then permits passage to the limit.
In each case, an almost-everywhere convergent subsequence preserves
the inequality.
\end{proof}
\begin{lemma}[Propagation of excess entropy production]
\label{lem:M-invariant}
Suppose \Cref{ass:CDkappa} holds. For either system,
\begin{equation}\label{eq:I-mono}
\sigma_{\mathrm{ex}}(t+s)
\le e^{2\kappa t}\sigma_{\mathrm{ex}}(s)<\infty,
\qquad s,t\ge0.
\end{equation}
In particular,
$t\mapsto e^{-2\kappa t}\sigma_{\mathrm{ex}}(t)$
is nonincreasing.
\end{lemma}

\begin{proof}
For a normalized initial profile $f_0=g^2/m$, the chain rule gives
\[
\mathcal E(f_0)
\le \frac{4\|g\|_\infty^2}{m^2}\mathcal E(g)<\infty.
\]
Thus \Cref{lem:gradient-form-extension} applies to both $f_0$ and its bounded evolutions $P_sf_0$.

For bounded $h\ge0$ with
$\sqrt h\in\mathcal D(\mathcal E)$,
the gradient estimate, chain rule and Markov Cauchy--Schwarz give
\[
\Gamma(P_th)
\le4e^{2\kappa t}(P_th)P_t\Gamma(\sqrt h).
\]
Set $v:=P_th$ and $v_\delta:=\sqrt{v+\delta}-\sqrt\delta$.
By invariance,
\[
\mathcal E(v_\delta)
=\frac14\int_X\frac{\Gamma(v)}{v+\delta}\,d\pi
\le e^{2\kappa t}\int_X
\frac{v}{v+\delta}P_t\Gamma(\sqrt h)\,d\pi
\le e^{2\kappa t}\mathcal E(\sqrt h).
\]
Closedness as $\delta\downarrow0$ yields
$\mathcal E(\sqrt{P_th})
\le e^{2\kappa t}\mathcal E(\sqrt h)$.
The admissible initial data have
$\mathcal E(\sqrt{f_0})<\infty$.
Applying this estimate first to $f_0$ at time $s$ and then
to $P_sf_0$ at time $t$ proves \Cref{eq:I-mono}.
\end{proof}

\begin{lemma}[Positive-time continuity of excess entropy production]
\label{lem:I-continuity}
Under \Cref{ass:CDkappa},
$\sigma_{\mathrm{ex}}^\eps$ and $\bar\sigma_{\mathrm{ex}}$
are continuous on $(0,\infty)$.
\end{lemma}

\begin{proof}
Work with either system and set
\[
a_s:=\frac{\Gamma(u_s)}{u_s},\qquad s\ge0.
\]
Admissibility gives
$a_0=4\Gamma(\sqrt{f_0})\in L^1(\pi)$.
The gradient estimate and Markov Cauchy--Schwarz imply
\[
0\le a_s\le e^{2\kappa s}P_sa_0,
\qquad
\sigma_{\mathrm{ex}}(s)=\int_Xa_s\,d\pi.
\]
If $t_j\to t>0$, interior parabolic regularity and positivity give
$a_{t_j}\to a_t$ locally uniformly.
By $L^2(\pi)$ strong continuity, $L^1(\pi)$ contractivity and density,
the semigroup is strongly continuous on $L^1(\pi)$.
Hence $P_{t_j}a_0\to P_ta_0$ in $L^1(\pi)$, and the displayed bound
makes $(a_{t_j})$ uniformly integrable.
Vitali's theorem proves the assertion.
\end{proof}
\begin{lemma}[Weighted convergence of projected coefficients]
\label{lem:weighted-cross-quadratic-conv}
Suppose \Cref{ass:standing,ass:coeff-weak} hold.
Fix an admissible pair $(f,\bar f)$ and $t>0$.
Set $A_\Phi^\eps:=D\Phi A^\eps D\Phi^\top$.
For every $\psi\in C_b(\bar E;\mathbb R^n)$,
\begin{align*}
\int_Eu^\eps(t)
\langle\psi\circ\Phi,D\Phi\gamma^\eps\rangle\,d\pi^\eps
&\longrightarrow
\int_{\bar E}\bar u(t)\langle\psi,\bar\gamma\rangle\,d\bar\pi,\\
\int_Eu^\eps(t)|\psi\circ\Phi|_{A_\Phi^\eps}^2\,d\pi^\eps
&\longrightarrow
\int_{\bar E}\bar u(t)|\psi|_{\bar A}^2\,d\bar\pi.
\end{align*}
\end{lemma}
\begin{proof}
Put
\[
r^\eps:=u^\eps(t)-\bar u(t)\circ\Phi,
\qquad
\mu^\eps:=\operatorname{tr}(A_\Phi^\eps)\pi^\eps.
\]
Normalization and Markov contractivity give $|r^\eps|\le2M$
for some $M<\infty$ and all sufficiently small $\eps$.
By \Cref{ass:coeff-weak}, the measures $\mu^\eps$ are tight
with uniformly bounded mass.
Since $r^\eps\to0$ locally uniformly, a compact--tail decomposition gives
\begin{equation}
    \int_E|r^\eps|^p\,d\mu^\eps\longrightarrow0,
\qquad p=1,2.
\label{eq:L2-alter}
\end{equation}

The errors from replacing $u^\eps(t)$ by $\bar u(t)\circ\Phi$
in the cross and quadratic terms are bounded, respectively, by
\[
\|\psi\|_\infty
\left(
\mathcal J_{\mathrm{hk,proj}}^\eps
\int_E|r^\eps|^2\,d\mu^\eps
\right)^{1/2},
\qquad
\|\psi\|_\infty^2
\int_E|r^\eps|\,d\mu^\eps,
\]
and hence vanish.
The remaining limits follow from the projected-current and
projected-diffusivity convergences in \Cref{ass:coeff-weak},
using the bounded continuous fields $\bar u(t)\psi$ and
$\sqrt{\bar u(t)}\psi$, respectively.
\end{proof}

\begin{lemma}[From curvature bounds to gradient estimates]
\label{lem:curvature-to-gradient}
Under the standing assumptions, suppose that the diffusion has
infinite lifetime in $X$ and that, for every
$\varphi\in C_c^\infty(X)$, the orbit $P_s\varphi$ is locally
$C^{1,3}$ on $(0,\infty)\times X$, converges to $\varphi$ in
$C^1_{\mathrm{loc}}(X)$ as $s\downarrow0$, and satisfies
\begin{equation}
\label{eq:curvature-orbit-bound}
\sup_{0\le s\le T}\|\Gamma(P_s\varphi)\|_\infty<\infty,
\qquad T>0.
\end{equation}
If $CD(-\kappa,\infty)$ holds for some $\kappa\ge0$, then
\begin{equation}
\label{eq:curvature-semigroup-gradient}
\sqrt{\Gamma(P_t\varphi)}
\le e^{\kappa t}P_t\sqrt{\Gamma(\varphi)},
\qquad t\ge0,\quad \varphi\in C_c^\infty(X).
\end{equation}
The estimate extends $\pi$-almost everywhere to bounded functions
in $\mathcal D(\mathcal E)$.

On $X=\mathbb R^d$, the lifetime and orbit hypotheses hold for
$\mathcal L=b\cdot\nabla+A:D^2$ if $A(\zeta)$ is positive definite
for every $\zeta$, $b\in C^2$ has bounded first and second
derivatives, and $\eta=(2A)^{1/2}\in C_b^2$.
\end{lemma}

\begin{proof}
Put $G=A^{-1}$. The covariant Hessian
$H_h:=\operatorname{Hess}_G h$ is given in coordinates by
\[
(H_h)_{ij}
=
\partial_{ij}h
-\frac12\sum_{k,\ell}A_{k\ell}
\bigl(
\partial_iG_{j\ell}
+\partial_jG_{i\ell}
-\partial_\ell G_{ij}
\bigr)\partial_kh,
\]
with squared Hilbert--Schmidt norm
$\|H_h\|_{\mathrm{HS},G}^2=\operatorname{tr}(AH_hAH_h)$.
The local Bochner identity takes the form
\[
\Gamma_2(h)
=
\|H_h\|_{\mathrm{HS},G}^2
+\nabla h^\top R\nabla h,
\]
where $R$ is a symmetric matrix field determined by the
coefficients of $\mathcal L$.
At a fixed point, a quadratic polynomial times a cutoff can
prescribe $\nabla h$ and $D^2h$ independently.
Testing $CD(-\kappa,\infty)$ with arbitrary gradient and
$H_h=0$ therefore gives $R\succeq-\kappa A$.
Moreover, $\nabla\Gamma(h)=2H_hA\nabla h$, so
Cauchy--Schwarz yields
\[
\Gamma(\Gamma(h))
\le
4\Gamma(h)\|H_h\|_{\mathrm{HS},G}^2
\le
4\Gamma(h)\bigl(\Gamma_2(h)+\kappa\Gamma(h)\bigr).
\]
These local identities and inequalities are understood almost
everywhere under the standing coefficient regularity and apply
to the semigroup orbits by locality.

For $\delta>0$, put
$h_s:=P_s\varphi$ and
$w_\delta(s,\zeta):=\sqrt{\Gamma(h_s)(\zeta)+\delta}$.
The chain rule, the preceding inequality and
$CD(-\kappa,\infty)$ give
\begin{align*}
(\partial_s-\mathcal L)w_\delta
&=
-\frac{\Gamma_2(h_s)}{w_\delta}
+\frac{\Gamma(\Gamma(h_s))}{4w_\delta^3}
\le
\frac{-\delta\Gamma_2(h_s)+\kappa\Gamma(h_s)^2}
     {w_\delta^3}
\le
\frac{\kappa\Gamma(h_s)}{w_\delta}
\le \kappa w_\delta.
\end{align*}

Fix $t>0$ and let $Z_r$ be the diffusion started at $\zeta$.
The localized It\^o--Krylov formula makes
\[
e^{\kappa r}w_\delta(t-r,Z_r),\qquad 0\le r<t,
\]
a local submartingale.
Nonexplosion and \Cref{eq:curvature-orbit-bound} allow removal
of compact-exit stopping by bounded convergence.
Path continuity and the initial $C^1_{\mathrm{loc}}$ convergence
then permit $r\uparrow t$, giving
\[
\sqrt{\Gamma(P_t\varphi)+\delta}
\le
e^{\kappa t}P_t\sqrt{\Gamma(\varphi)+\delta}.
\]
Letting $\delta\downarrow0$ proves
\Cref{eq:curvature-semigroup-gradient}.
The form-domain extension follows from
\Cref{lem:gradient-form-extension}.

Under the stated coefficient conditions, global Lipschitz
continuity and Kunita's flow theorem give nonexplosion and
a $C^1$ flow whose Jacobian has uniformly bounded second moments
over initial points and bounded time intervals.
Flow differentiation and boundedness of $A$ yield
\Cref{eq:curvature-orbit-bound} and the initial
$C^1_{\mathrm{loc}}$ convergence.
Local ellipticity and interior Schauder estimates give the
positive-time $C^{1,3}_{\mathrm{loc}}$ regularity.
\end{proof}
\section{Proofs for Results in Section 4}
\subsection{Slow--fast averaging}
\begin{lemma}[Continuity of the averaging projection]
\label{lem:P-continuity}
If $y\mapsto\mu^y$ is weakly continuous, then
$\mathcal P:C_b(E)\to C_b(\bar E)$.
\end{lemma}
\begin{proof}%[Proof of \Cref{lem:P-continuity}]
For $f\in C_b(E)$, one has
$\|\mathcal Pf\|_\infty\le\|f\|_\infty$.
If $y_k\to y$, weak continuity gives
$\mu^{y_k}\Rightarrow\mu^y$ and uniform tightness.
Together with the uniform convergence of
$f(\cdot,y_k)$ to $f(\cdot,y)$ on compact sets,
a compact--tail decomposition yields
\[
\int f(x,y_k)\,\mu^{y_k}(dx)
\longrightarrow
\int f(x,y)\,\mu^y(dx).
\]
Thus $\mathcal Pf$ is continuous.
\end{proof}

\subsubsection{Slow--fast OU process}
\label{subsubsec:OU}
We consider the OU model in \Cref{ex:OU}. Throughout, assume that
$B_{11}$ and $C:=B_{22}-B_{21}B_{11}^{-1}B_{12}$ are positive stable.
Set
\begin{equation}\label{eq:OU-fast-equilibrium}
H:=B_{11}^{-1}B_{12},
\end{equation}
and let $\Sigma^x,\Sigma^y\succ0$ be the unique solutions of
\begin{equation}\label{eq:OU-limit-Lyapunov}
B_{11}\Sigma^x+\Sigma^xB_{11}^{\top}=2I_{d_x},
\qquad
C\Sigma^y+\Sigma^yC^{\top}=2I_{d_y}.
\end{equation}
The frozen invariant measure and averaged coefficients are
\[
\mu^y=\mathcal N(-Hy,\Sigma^x),\qquad
\bar A=I_{d_y},\qquad
\bar b(y)=-Cy,\qquad
\bar\pi=\mathcal N(0,\Sigma^y).
\]
In the notation of \Cref{sec:setting}, this gives
\[
\bar K:=(\Sigma^y)^{-1}-C,\qquad
\bar\gamma(y)=\bar F(y)=\bar Ky,\qquad
\bar\sigma_{\mathrm{ss}}
=\tr(\bar K\Sigma^y\bar K^\top)<\infty.
\]

We first establish the invariant-measure and coefficient limits,
then characterize locking by \Cref{prop:direct-force-criterion}.
Initial-data compatibility is needed only for the subsequent
dynamical statement.

Throughout the proofs below, matrix norms are Euclidean operator
norms, and constants $c,c_T$ are independent of sufficiently
small $\eps$.

\begin{lemma}[Invariant measures and coefficient convergence]
\label{lem:OU-static}
For some $\eps_0>0$, the OU process has an invariant measure
$\pi^\eps=\mathcal N(0,\Sigma^\eps)$ for every
$0<\eps<\eps_0$, and
\[
\pi^\eps\Rightarrow\Pi,\qquad
\Pi(dx,dy)=\mu^y(dx)\bar\pi(dy).
\]
The family satisfies \Cref{ass:coeff-weak}.
\end{lemma}

\begin{proof}
Set $M^\eps:=I^\eps B$. We seek a positive definite solution of
\begin{equation}\label{eq:OU-Lyapunov-compact}
M^\eps\Sigma^\eps+\Sigma^\eps(M^\eps)^\top=2I^\eps.
\end{equation}
Lyapunov's theorem then gives positive stability of $M^\eps$
and hence the invariant Gaussian law.

The coordinates $w=x+Hy$ center the fast variable at its frozen
equilibrium. The transformed drift and diffusion matrices are
\begin{equation}\label{eq:OU-transformed-drift}
\widetilde M^\eps=
\begin{pmatrix}
\eps^{-1}B_{11}+HB_{21}&HC\\
B_{21}&C
\end{pmatrix},
\qquad
\widetilde A^\eps=
\begin{pmatrix}
\eps^{-1}I_{d_x}+HH^\top&H\\
H^\top&I_{d_y}
\end{pmatrix}.
\end{equation}
Writing the unknown symmetric matrix in these coordinates as
\[
\widetilde\Sigma^\eps=
\begin{pmatrix}
U^\eps&V^\eps\\
(V^\eps)^\top&Y^\eps
\end{pmatrix},
\]
the transformed Lyapunov equation becomes
\begin{align}
B_{11}V^\eps
&=\eps\bigl(
2H-HB_{21}V^\eps-HCY^\eps
-U^\eps B_{21}^\top-V^\eps C^\top
\bigr),
\label{eq:OU-cov-TR-compact}\\
CY^\eps+Y^\eps C^\top
&=2I_{d_y}-B_{21}V^\eps-(V^\eps)^\top B_{21}^\top,
\label{eq:OU-cov-BR-compact}\\
B_{11}U^\eps+U^\eps B_{11}^\top
&=2I_{d_x}
+2\eps\bigl[
HH^\top-\Sym\bigl(HB_{21}U^\eps+HC(V^\eps)^\top\bigr)
\bigr].
\label{eq:OU-cov-TL-compact}
\end{align}
This finite-dimensional linear system has coefficients continuous
at $\eps=0$. Its limiting system is triangular and invertible,
since $B_{11}$ and the Lyapunov operators associated with
$B_{11}$ and $C$ are invertible.
Continuity of the inverse therefore yields
\[
(U^\eps,V^\eps,Y^\eps)\longrightarrow(\Sigma^x,0,\Sigma^y).
\]
Thus $\widetilde\Sigma^\eps\succ0$ for sufficiently small $\eps$.
Returning to the original coordinates gives
\begin{equation}\label{eq:Sigma0-OU}
\Sigma^\eps\longrightarrow\Sigma^0:=
\begin{pmatrix}I&-H\\0&I\end{pmatrix}
\begin{pmatrix}\Sigma^x&0\\0&\Sigma^y\end{pmatrix}
\begin{pmatrix}I&0\\-H^\top&I\end{pmatrix}
\succ0.
\end{equation}
Consequently, the invariant laws
$\pi^\eps=\mathcal N(0,\Sigma^\eps)$ converge to
$\Pi=\mathcal N(0,\Sigma^0)$, with convergence of second moments.
The factorization shows that $W=X+HY$ and $Y$ are independent
under $\Pi$, with covariances $\Sigma^x$ and $\Sigma^y$.
Hence $\Pi(dx,dy)=\mu^y(dx)\bar\pi(dy)$.

For the coefficient limits, set
\[
K^\eps:=(\Sigma^\eps)^{-1}-B,\qquad
K^0:=(\Sigma^0)^{-1}-B.
\]
Then $K^\eps\to K^0$ and
\[
\gamma^\eps(z)=I^\eps K^\eps z,\qquad
F^\eps(z)=K^\eps z,\qquad
D\Phi\,\gamma^\eps(z)=(K^\eps z)_y.
\]
Uniform second-moment bounds give uniform integrability of the
linear projected currents. Thus, for every
$\xi\in C_b(\bar E;\mathbb R^{d_y})$,
\[
\int_E\langle\xi(y),(K^\eps z)_y\rangle\,d\pi^\eps
\longrightarrow
\int_E\langle\xi(y),(K^0z)_y\rangle\,d\Pi
=
\int_{\bar E}\langle\xi(y),\bar Ky\rangle\,d\bar\pi.
\]
The last equality uses
$\mathbb E_\Pi[Z\mid Y=y]=(-Hy,y)$ and
$(K^0(-Hy,y))_y=\bar Ky$.
The remaining clauses of \Cref{ass:coeff-weak} follow from
\[
Q^\eps=I_{d_y}\pi^\eps\Rightarrow Q:=I_{d_y}\Pi,
\qquad
\Phi_\#Q=I_{d_y}\bar\pi=\bar A\bar\pi,
\]
and the bound
\[
\sup_{0<\eps<\eps_0}\mathcal J_{\mathrm{hk,proj}}^\eps
=
\sup_{0<\eps<\eps_0}
\int_E|(K^\eps z)_y|^2\,d\pi^\eps<\infty.
\]
\end{proof}
\begin{lemma}[OU profile compatibility]
\label{lem:OU-profile-compatibility}
In the OU setting above, let $\mathcal M_{\mathrm{OU}}$
be the profile class introduced in \Cref{subsubsec:OU-linear}. Then
\[
1\in\mathcal M_{\mathrm{OU}}\subset\mathcal M,
\qquad
\mathcal P(\mathcal M_{\mathrm{OU}})
\subset\bar{\mathcal M}.
\]
\end{lemma}

\begin{proof}
Fix $g^2\in\mathcal M_{\mathrm{OU}}$.
For each $0<\varepsilon<\varepsilon_0$, Gaussian moments give
$\int_E\Gamma^\varepsilon(g)\,d\pi^\varepsilon<\infty$.
Since $A^\varepsilon$ is constant, smooth cutoff approximation
therefore gives $g^2\in\mathcal M$.

Set $\bar g:=\sqrt{\mathcal P(g^2)}$, which is bounded and
strictly positive.
The change of variables $x=w-Hy$ in the frozen Gaussian integral,
together with the polynomial derivative bounds, justifies
differentiation of every order. Thus $\bar g$ is smooth, with
\[
\nabla_y\bar g(y)
=
\frac{1}{
\bar g(y)
}
\displaystyle
\int_{\mathbb R^{d_x}}
g(x,y)
\bigl(\nabla_y g(x,y)-H^\top\nabla_x g(x,y)\bigr)
\,\mu^y(dx)
.
\]
Writing $\Pi(dx,dy)=\mu^y(dx)\bar\pi(dy)$,
Cauchy--Schwarz and Gaussian moments yield
\[
\int_{\bar E}|\nabla_y\bar g|^2\,d\bar\pi
\le
\int_E
|\nabla_y g-H^\top\nabla_x g|^2\,d\Pi
<\infty.
\]
Since $\bar A=I_{d_y}$, smooth cutoff approximation gives
$\bar g^2=\mathcal P(g^2)\in\bar{\mathcal M}$.
Finally, $1\in\mathcal M_{\mathrm{OU}}$ follows by taking $g=1$.
\end{proof}
\begin{lemma}[Sharp characterization of stationary OU locking]
\label{lem:OU-locking}
Under the preceding stability assumptions, \Cref{ass:locking}
holds if and only if
\[
B_{11}=B_{11}^\top,\qquad B_{21}=B_{12}^\top.
\]
\end{lemma}

\begin{proof}
By \Cref{lem:OU-static},
$\mathbb E_{\pi^\eps}|\bar KY|^2\to\bar\sigma_{\mathrm{ss}}$,
so \Cref{prop:direct-force-criterion} applies.
Set
\[
D^\eps:=K^\eps-K^0\longrightarrow0,\qquad
\Delta_{\mathrm f}:=(\Sigma^x)^{-1}-B_{11},\qquad
\Delta_{\mathrm s}:=H^\top(\Sigma^x)^{-1}-B_{21}.
\]
The factorization in \Cref{eq:Sigma0-OU} gives, with $w=x+Hy$,
\begin{equation}\label{eq:OU-locking-obstructions}
K^\eps z-\binom{0}{\bar Ky}
=
\binom{\Delta_{\mathrm f}w}{\Delta_{\mathrm s}w}
+D^\eps z.
\end{equation}
Since $A^\eps\succeq I$ for $\eps\le1$, locking and convergence
of the quadratic moments imply
\[
\mathbb E_\Pi
\bigl(|\Delta_{\mathrm f}W|^2+|\Delta_{\mathrm s}W|^2\bigr)=0.
\]
As $\operatorname{Cov}_\Pi(W)=\Sigma^x\succ0$, this requires
$\Delta_{\mathrm f}=\Delta_{\mathrm s}=0$.

Conversely, suppose $\Delta_{\mathrm f}=\Delta_{\mathrm s}=0$.
Then \Cref{eq:OU-locking-obstructions} gives
$K^0z=\binom{0}{\bar Ky}$, hence $A^\eps K^0=K^0$.
The matrix $D^\eps=(\Sigma^\eps)^{-1}-(\Sigma^0)^{-1}$
is symmetric, whereas $A^\eps K^\eps\Sigma^\eps$
is skew-symmetric by \Cref{eq:OU-Lyapunov-compact}.
Consequently,
\[
\begin{aligned}
\mathbb E_{\pi^\eps}|D^\eps Z|_{A^\eps}^2
&=
\tr\bigl(D^\eps A^\eps K^\eps\Sigma^\eps\bigr)
-
\tr\bigl(D^\eps K^0\Sigma^\eps\bigr)\\
&=
-\tr\bigl(D^\eps K^0\Sigma^\eps\bigr)
\longrightarrow0,
\end{aligned}
\]
since $D^\eps\to0$ and $\Sigma^\eps\to\Sigma^0$.
This proves locking.

Finally, \Cref{eq:OU-limit-Lyapunov} and positive stability give
$\Delta_{\mathrm f}=0$ if and only if $B_{11}=B_{11}^\top$.
In this case $\Sigma^x=B_{11}^{-1}$ and
$\Delta_{\mathrm s}=B_{12}^\top-B_{21}$, proving the claim.
\end{proof}

\begin{lemma}[Dynamic convergence and cutoffs]
\label{lem:standing}
\Cref{ass:standing} holds with microscopic profile class
$\mathcal M_{\mathrm{OU}}$, and $(f,\mathcal Pf)$ is admissible
for every $f\in\mathcal M_{\mathrm{OU}}$.
\end{lemma}

\begin{proof}
The invariant-measure clause follows from \Cref{lem:OU-static}.
Fix $f\in\mathcal M_{\mathrm{OU}}$ and put
$\bar f:=\mathcal Pf\in\bar{\mathcal M}$.
The invariant-measure limit and its disintegration give
\[
m^\eps=\int_E f\,d\pi^\eps
\longrightarrow
\int_E f\,d\Pi
=\int_{\bar E}\bar f\,d\bar\pi=\bar m>0.
\]

Let $E_s^\eps=e^{-M^\eps s}$,
$E_s^\eps z=(x_s^\eps,y_s^\eps)$, and
$w_s^\eps=x_s^\eps+Hy_s^\eps$.
By \Cref{eq:OU-Lyapunov-compact},
\[
\Sigma^\eps-E_s^\eps\Sigma^\eps(E_s^\eps)^\top
=
2\int_0^s E_r^\eps A^\eps(E_r^\eps)^\top\,dr
\succeq0.
\]
Since $\Sigma^\eps\to\Sigma^0\succ0$,
\[
\|E_s^\eps\|^2
\le
\|\Sigma^\eps\|\,\|(\Sigma^\eps)^{-1}\|
\le c,
\qquad s\ge0.
\]
Consequently,
$|w_s^\eps|+|y_s^\eps|\le c|z|$ for $s\ge0$.
By \Cref{eq:OU-transformed-drift}, variation of constants yields
\[
\begin{aligned}
w_s^\eps
&=e^{-B_{11}s/\eps}(x+Hy)
-\int_0^s e^{-B_{11}(s-r)/\eps}
H(B_{21}w_r^\eps+Cy_r^\eps)\,dr,\\
y_s^\eps
&=e^{-Cs}y
-\int_0^s e^{-C(s-r)}B_{21}w_r^\eps\,dr.
\end{aligned}
\]
Choose $a>0$ such that
$\|e^{-B_{11}s}\|\le ce^{-as}$.
Using the uniform bound above and
$\int_0^s e^{-a(s-r)/\eps}\,dr\le\eps/a$,
the first formula, followed by time integration in the second,
gives, for $0\le s\le T$,
\begin{align}
|w_s^\eps|
&\le c\bigl(e^{-as/\eps}|x+Hy|+\eps|z|\bigr),
\label{eq:w-Oeps}\\
|y_s^\eps-e^{-Cs}y|
&\le c_T\eps|z|.
\label{eq:y-Oeps}
\end{align}
Consequently, for every $t>0$, in operator norm,
\begin{equation}\label{eq:OU-mean-limit-compact}
E_t^\eps\longrightarrow L_t,\qquad
L_t(x,y)=(-He^{-Ct}y,e^{-Ct}y).
\end{equation}

The Gaussian transition formula is
\begin{equation}\label{eq:OU-transition-kernel-compact}
P_t^\eps f_0^\eps(z)
=
\int_E f_0^\eps\,d\mathcal N(E_t^\eps z,\mathsf C_t^\eps),
\qquad
\mathsf C_t^\eps
=
\Sigma^\eps-E_t^\eps\Sigma^\eps(E_t^\eps)^\top.
\end{equation}
Set
\[
\Sigma_t^y
:=
\Sigma^y-e^{-Ct}\Sigma^ye^{-C^\top t}
=
2\int_0^t e^{-Cs}e^{-C^\top s}\,ds\succ0.
\]
By \Cref{eq:Sigma0-OU,eq:OU-mean-limit-compact},
\begin{equation}\label{eq:OU-transition-covariance}
\mathsf C_t^\eps\longrightarrow
\mathsf C_t^0:=
\begin{pmatrix}
\Sigma^x+H\Sigma_t^yH^\top&-H\Sigma_t^y\\
-\Sigma_t^yH^\top&\Sigma_t^y
\end{pmatrix}
\succ0.
\end{equation}
The limiting Gaussian law has the averaged OU slow marginal
and conditional fast law $\mu^{y'}$ given $Y=y'$.
Therefore,
\begin{equation}\label{eq:OU-limit-kernel-identification}
\frac1{\bar m}
\int_E f\,d\mathcal N(L_tz,\mathsf C_t^0)
=
\bar P_t\!\left(\frac{\mathcal Pf}{\bar m}\right)(y)
=
(\bar P_t\bar f_0)(y).
\end{equation}

The means converge locally uniformly and the limiting covariance
is positive definite. The Gaussian relative-entropy formula and
Pinsker's inequality therefore give locally uniform convergence
of the transition laws in total variation.
Since $f$ is bounded and $m^\eps\to\bar m>0$, it follows that,
for every compact $K\subset E$,
\begin{equation}\label{eq:OU-ORBIT-LIMIT}
\sup_{z\in K}
\left|
P_t^\eps f_0^\eps(z)
-
(\bar P_t\bar f_0)(\Phi(z))
\right|
\longrightarrow0.
\end{equation}
This proves \Cref{ass:standing}(ii).

Finally, choose smooth radial cutoffs $0\le\chi_R\uparrow1$,
equal to one on $B_R$ and supported in $B_{2R}$, with
$|\nabla\chi_R|\le c/R$ and $|D^2\chi_R|\le c/R^2$.
The linear drift and constant diffusion give, for $R\ge1$
and each fixed $0<\eps<\eps_0$,
\[
\int_E
\bigl(|\mathcal L^\eps\chi_R|+\Gamma^\eps(\chi_R)\bigr)
\,d\pi^\eps
\le c_\eps\,\pi^\eps(|z|\ge R)
\longrightarrow0,
\]
where $c_\eps$ is independent of $R$.
The same construction applies to the macroscopic process,
proving \Cref{ass:standing}(iii).
\end{proof}
\begin{lemma}[Uniform gradient estimate]
\label{lem:OU-CD}
Assume $\Sym(B_{11})\succ0$, write
\[
\Sym(B)=
\begin{pmatrix}S_{11}&S_{12}\\S_{21}&S_{22}\end{pmatrix},
\qquad
S_{\mathrm{sc}}:=S_{22}-S_{21}S_{11}^{-1}S_{12},
\]
and set $\kappa:=\max\{0,-\lambda_{\min}(S_{\mathrm{sc}})\}$.
Then \Cref{ass:CDkappa} holds with this $\kappa$.
\end{lemma}

\begin{proof}
The Schur complement controls the coupling by the positive fast
block. We first obtain the resulting bound independently of
$\eps$, and then use the explicit OU derivative formula to
deduce the required gradient estimate.

Completing the square gives, for $u=I^\eps v$,
\[
u^\top\Sym(B)u
\ge-\kappa|u_y|^2
\ge-\kappa v^\top I^\eps v.
\]
Since $M^\eps=I^\eps B$, this is equivalent to
\[
M^\eps I^\eps+I^\eps(M^\eps)^\top+2\kappa I^\eps\succeq0.
\]
Consequently,
\[
\frac{d}{dt}
\left(e^{-2\kappa t}E_t^\eps I^\eps(E_t^\eps)^\top\right)
=
-e^{-2\kappa t}E_t^\eps
\left(M^\eps I^\eps+I^\eps(M^\eps)^\top+2\kappa I^\eps\right)
(E_t^\eps)^\top
\preceq0.
\]
Evaluating at $t=0$ yields
\[
E_t^\eps I^\eps(E_t^\eps)^\top
\preceq e^{2\kappa t}I^\eps.
\]
For $\varphi\in C_c^\infty(E)$, differentiation of the Gaussian
transition formula gives
$\nabla P_t^\eps\varphi=(E_t^\eps)^\top P_t^\eps\nabla\varphi$.

The matrix bound and Jensen's inequality therefore imply
\[
\sqrt{\Gamma^\eps(P_t^\eps\varphi)}
\le e^{\kappa t}
\big|P_t^\eps\nabla\varphi\big|_{I^\eps}
\le e^{\kappa t}P_t^\eps\sqrt{\Gamma^\eps(\varphi)}.
\]

For the macroscopic estimate, evaluating the quadratic form of
$\Sym(B)$ at $(-Hy,y)$ gives
\[
\Sym(C)
=
S_{\mathrm{sc}}
+(S_{11}^{-1}S_{12}-H)^\top
S_{11}(S_{11}^{-1}S_{12}-H)
\succeq-\kappa I.
\]
The same argument with drift matrix $C$ and diffusion matrix $I$
then gives
\[
\sqrt{\bar\Gamma(\bar P_t\phi)}
\le e^{\kappa t}\bar P_t\sqrt{\bar\Gamma(\phi)},
\qquad \phi\in C_c^\infty(\bar E).
\]

Applying \Cref{lem:gradient-form-extension} to both semigroups proves
\Cref{ass:CDkappa}.
\end{proof}
\subsubsection{Uniform curvature criteria for averaging diffusions}
\label{subsubsec:Ito-Kunita}

We give sufficient conditions for the microscopic gradient estimate in
\Cref{ass:CDkappa} for the averaging diffusion
\Cref{eq:NL-avg-SDE}.

The standing semigroup, core, form, and regularity assumptions remain
in force.
Under the common-invariant-measure hypotheses of
\Cref{lem:avg-gradient-inheritance}, the microscopic estimate passes to
the averaged semigroup with the same constant. 

%\subsubsection{Spatially constant diffusion}

\begin{theorem}[Microscopic Schur criterion]
\label{thm:CD_schur_averaging}
Suppose that $A^\eps$ is spatially
constant. Write $b=(b_1^\top,b_2^\top)^\top$ and set
\[
\begin{gathered}
Db=
\begin{pmatrix}
D_xb_1&D_yb_1\\
D_xb_2&D_yb_2
\end{pmatrix},
\\
\mathsf S:=-\operatorname{Sym}\!\left(
\diag(a_1,a_2)^{-1/2}Db\,\diag(a_1,a_2)^{1/2}
\right)
=\begin{pmatrix}
\mathsf S_{11}&\mathsf S_{12}\\
\mathsf S_{21}&\mathsf S_{22}
\end{pmatrix}.
\end{gathered}
\]
If, for some $\kappa_{\rm mic}\ge0$,
\[
\begin{gathered}
\mathsf S_{11}(z)\succ0,\qquad
\mathsf S_{22}(z)-\mathsf S_{21}(z)
\mathsf S_{11}(z)^{-1}\mathsf S_{12}(z)
\succeq-\kappa_{\rm mic}I,
\qquad z\in E,
\end{gathered}
\]
then $\mathcal L^\eps$ satisfies $CD(-\kappa_{\rm mic},\infty)$
for every $\eps>0$, and the microscopic gradient estimate in
\Cref{ass:CDkappa} holds with the same constant.
\end{theorem}

\begin{proof}
Since $Db^\eps=\diag(\eps^{-1}I_{d_x},I_{d_y})Db$, the normalized
curvature matrix is
\[
(A^\eps)^{-1/2}
\left[-\operatorname{Sym}\!\left((Db^\eps)A^\eps\right)\right]
(A^\eps)^{-1/2}
=
\begin{pmatrix}
\eps^{-1}\mathsf S_{11}&\eps^{-1/2}\mathsf S_{12}\\
\eps^{-1/2}\mathsf S_{21}&\mathsf S_{22}
\end{pmatrix}.
\]
After adding $\kappa_{\rm mic}I$, its Schur complement is bounded
below by
\[
\mathsf S_{22}+\kappa_{\rm mic}I
-\mathsf S_{21}\mathsf S_{11}^{-1}\mathsf S_{12}\succeq0,
\]
because
$(\mathsf S_{11}+\eps\kappa_{\rm mic}I)^{-1}
\preceq\mathsf S_{11}^{-1}$. Thus
\[
-\operatorname{Sym}\!\left((Db^\eps)A^\eps\right)
+\kappa_{\rm mic}A^\eps\succeq0,
\]
which gives the asserted curvature bound.

For the additive-noise flow, put $G^\eps=(A^\eps)^{-1}$ and
$J_t^\eps(z)=D_zZ_t^{\eps,z}$. Local flow differentiation and
nonexplosion give
\[
\frac{d}{dt}J_t^\eps(z)
=(Db^\eps)(Z_t^{\eps,z})J_t^\eps(z).
\]
The same matrix inequality implies
\[
\frac{d}{dt}|J_t^\eps(z)v|_{G^\eps}^2
\le2\kappa_{\rm mic}|J_t^\eps(z)v|_{G^\eps}^2.
\]
Gronwall's inequality gives
$|J_t^\eps(z)v|_{G^\eps}
\le e^{\kappa_{\rm mic}t}|v|_{G^\eps}$.
This bound justifies differentiation under the expectation and yields
\[
|D(P_t^\eps\varphi)(z)v|
\le e^{\kappa_{\rm mic}t}|v|_{G^\eps}
P_t^\eps\sqrt{\Gamma^\eps(\varphi)}(z),
\qquad \varphi\in C_c^\infty(E).
\]
Taking the supremum over $|v|_{G^\eps}=1$ proves the gradient
estimate. Its form-domain extension follows from
\Cref{lem:gradient-form-extension}.
\end{proof}

%\subsubsection{Multiplicative noise: coefficient conditions and proof}

\begin{assumptionx}[Weighted variational-flow structure]
\label{ass:IKB-condensed}
In \Cref{eq:NL-avg-SDE}, take $a_2=a_2(y)$ and write
\[
\begin{gathered}
B_i:=a_i^{-1},\qquad
\eta_i:=(2a_i)^{1/2},\qquad i=1,2,\qquad
A^\eps:=\diag(\eps^{-1}a_1,a_2),\qquad
G^\eps:=(A^\eps)^{-1}.
\end{gathered}
\]
The Brownian blocks $W^{(1)}$ and $W^{(2)}$ are independent, so
\[
\begin{aligned}
dX_t^\eps
&=\eps^{-1}b_1(Z_t^\eps)\,dt
+\eps^{-1/2}\eta_1(Z_t^\eps)\,dW_t^{(1)},&
dY_t^\eps
&=b_2(Z_t^\eps)\,dt
+\eta_2(Y_t^\eps)\,dW_t^{(2)}.
\end{aligned}
\]
All coefficient fields below are evaluated at $z=(x,y)$;
$B_2$ and $\eta_2$ depend only on $y$.
Differential operators act componentwise on matrix-valued fields.

Since $G^\eps=\operatorname{diag}(\eps B_1,B_2)$,
the derivative-flow energy weights fast variations by $\eps$.
The bounds below allow fast contraction to absorb the fast--slow
coupling terms, yielding an estimate uniform in sufficiently
small $\eps$.
Assume:
\begin{enumerate}
\item
There exist $0<\lambda_i\le\Lambda_i<\infty$ such that
\[
\lambda_i I\preceq a_i\preceq\Lambda_i I,\qquad i=1,2.
\]
The fields $b_1,b_2$ are $C^2$ with bounded first and second
derivatives,
$\eta_1\in C_b^2(\mathbb R^{d_x+d_y})$, and
$\eta_2\in C_b^2(\mathbb R^{d_y})$.
The standing semigroup, core, and regularity assumptions hold,
together with the cutoff condition in \Cref{ass:standing}(iii).

\item
Write $\eta_{1,k}$ for the columns of $\eta_1$, set
\[
\begin{aligned}
 P_k&:=D_x\eta_{1,k},\qquad
Q_k:=D_y\eta_{1,k},\qquad
N_k:=(D_xB_1)[\eta_{1,k}],\\
M_{1,f}
&:=b_1\cdot D_xB_1+a_1:D_x^2B_1,\qquad
M_{1,s}
:=b_2\cdot D_yB_1+a_2:D_y^2B_1.
\end{aligned}
\]
Define
\[
\begin{aligned}
K_{11}
&:=2\operatorname{Sym}(B_1D_xb_1)+M_{1,f}+\sum_k\left(
P_k^\top B_1P_k
+2\operatorname{Sym}(N_kP_k)
\right),\\
K_{12}
&:=B_1D_yb_1
+\sum_k\left(
P_k^\top B_1Q_k+N_kQ_k
\right),
\qquad K_{22}
:=\sum_k Q_k^\top B_1Q_k,\qquad\mathcal K_f
:=
\begin{pmatrix}
K_{11} & K_{12}\\
K_{12}^\top & K_{22}
\end{pmatrix}.
\end{aligned}
\]

\item

Define
\[
\begin{aligned}
R_\ell:=D_y\eta_{2,\ell},\quad
\Theta_\ell:=(D_yB_2)[\eta_{2,\ell}],\quad M_2:=b_2\cdot D_yB_2+a_2:D_y^2B_2,\quad
L_{12}
:=(D_xb_2)^\top B_2,\\
L_{22}
:=2\operatorname{Sym}(B_2D_yb_2)+M_2+\sum_\ell\left(
R_\ell^\top B_2R_\ell
+2\operatorname{Sym}(\Theta_\ell R_\ell)
\right),\quad\mathcal K_s
:=
\begin{pmatrix}
0 & L_{12}\\
L_{12}^\top & L_{22}
\end{pmatrix}.
\end{aligned}
\]
\item
The following constants are finite, with $\kappa_f>0$:
\begin{equation}
\label{eq:IKB-constants}
\begin{aligned}
\kappa_f
&:=-\sup_z\lambda_{\max}
\left(B_1^{-1/2}K_{11}B_1^{-1/2}\right),
&\ell_f&
:=\sup_z\left\|
B_1^{-1/2}K_{12}B_2^{-1/2}
\right\|_{\rm op},\\
\mu_f
&:=\sup_z\lambda_{\max}
\left(B_2^{-1/2}K_{22}B_2^{-1/2}\right),
&C_h
&:=\max\left\{
0,\sup_z\lambda_{\max}
\left(
B_1^{-1/2}\operatorname{Sym}(M_{1,s})B_1^{-1/2}
\right)
\right\},\\
\ell_s
&:=\sup_z\left\|
B_1^{-1/2}L_{12}B_2^{-1/2}
\right\|_{\rm op},&
\mu_s
&:=\sup_z\lambda_{\max}
\left(
B_2^{-1/2}\operatorname{Sym}(L_{22})B_2^{-1/2}
\right).
\end{aligned}
\end{equation}
Choose $\delta_f\in(0,\kappa_f)$ and $\delta_s>0$ such that
\[
\Delta:=\kappa_f-\delta_f-\delta_s>0.
\]
Set
\[
\begin{gathered}
\label{eq:IK-epsilon0}
\beta_0:=\mu_f+\frac{\ell_f^2}{\delta_f},\quad
\beta_s:=\mu_s+\frac{\ell_s^2}{\delta_s},\quad
\rho:=\frac{\beta_0+\beta_s}{2},\quad
\kappa_{\rm mic}:=\rho_+,\quad \eps_0:=1\wedge
\frac{\Delta}{2\bigl(C_h+(-2\rho)_+\bigr)}.
\end{gathered}
\]

Here $r_+:=\max\{r,0\}$, and the quotient is interpreted as
$+\infty$ when its denominator vanishes.
\end{enumerate}
\end{assumptionx}

\begin{theorem}[Microscopic curvature via the derivative flow]
\label{thm:CD-negative-short}
Under \Cref{ass:IKB-condensed}, $\mathcal L^\eps$ satisfies
$CD(-\kappa_{\rm mic},\infty)$ for $0<\eps\le\eps_0$, and
\[
\sqrt{\Gamma^\eps(P_t^\eps\varphi)}
\le e^{\kappa_{\rm mic}t}
P_t^\eps\sqrt{\Gamma^\eps(\varphi)},
\qquad t\ge0,\quad \varphi\in C_c^\infty(E).
\]
The estimate extends $\pi^\eps$-almost everywhere to
$\mathcal D(\mathcal E^\eps)\cap L^\infty(\pi^\eps)$.
\end{theorem}

\begin{proof}
Fix $0<\eps\le\eps_0$, $z\in E$ and
$v\in\mathbb R^{d_x+d_y}$.
Write $Z_t=Z_t^{\eps,z}$ and
$(\Xi_t,\Upsilon_t)=J_t^\eps(z)v$ for the fast and slow variations.
Their equations are
\[
\begin{aligned}
d\Xi_t
&=\eps^{-1}\bigl((D_xb_1)\Xi_t+(D_yb_1)\Upsilon_t\bigr)\,dt+\eps^{-1/2}\sum_k(P_k\Xi_t+Q_k\Upsilon_t)\,dW_t^{(1),k},\\
d\Upsilon_t
&=\bigl((D_xb_2)\Xi_t+(D_yb_2)\Upsilon_t\bigr)\,dt+\sum_\ell R_\ell\Upsilon_t\,dW_t^{(2),\ell},
\end{aligned}
\]
with all coefficients evaluated at $Z_t$.
Define
\[
h(t):=\mathbb E[\Xi_t^\top B_1(Z_t)\Xi_t],
\qquad
j(t):=\mathbb E[\Upsilon_t^\top B_2(Y_t^\eps)\Upsilon_t],\qquad
g_\eps(t)
:= \eps h(t)+j(t).
\]
The coefficient bounds give finite moments of the flow and its
Jacobian on finite time intervals. Localized It\^o's formula
therefore yields locally absolutely continuous $h,j$, with
\[
\begin{aligned}
h'
&=\eps^{-1}\mathbb E\!\left[
\Xi_t^\top K_{11}\Xi_t
+2\Xi_t^\top K_{12}\Upsilon_t
+\Upsilon_t^\top K_{22}\Upsilon_t\right]+\mathbb E[\Xi_t^\top M_{1,s}\Xi_t],\\
j'
&=\mathbb E\!\left[
2\Xi_t^\top L_{12}\Upsilon_t
+\Upsilon_t^\top L_{22}\Upsilon_t\right]
\end{aligned}
\]
almost everywhere. By \Cref{eq:IKB-constants} and Cauchy--Schwarz,
\begin{equation}
\label{eq:flow-energy-bounds}
\begin{aligned}
h'&\le\left(-\frac{\kappa_f}{\eps}+C_h\right)h
+\frac{2\ell_f}{\eps}\sqrt{hj}
+\frac{\mu_f}{\eps}j,\qquad
j'\le2\ell_s\sqrt{hj}+\mu_sj.
\end{aligned}
\end{equation}
Young's inequality, applied with
$\delta_f$ and $\delta_s$, gives
\begin{equation}
\label{eq:flow-energy-absorption}
g_\eps'
\le2\rho g_\eps+
\left[-\Delta+\eps(C_h-2\rho)\right]h.
\end{equation}
By the choice of $\varepsilon_0$ in \Cref{ass:IKB-condensed}(4),
the bracket is at most $-\Delta/2$.
Gronwall's inequality gives $g_\eps(t)\le e^{2\rho t}g_\eps(0)$.
Since $\rho\le\kappa_{\rm mic}$, this also gives
$g_\eps(t)\le e^{2\kappa_{\rm mic}t}g_\eps(0)$.

Since
\[
\begin{gathered}
g_\eps(t)=\mathbb E\!\left[
(J_t^\eps(z)v)^\top G^\eps(Z_t)J_t^\eps(z)v\right],\qquad
g_\eps(0)=v^\top G^\eps(z)v,
\end{gathered}
\]
this is \Cref{eq:B2-Jacobian-energy}.
\Cref{thm:flow-CD} completes the proof.
\end{proof}
\subsubsection{Nonlinear averaging with a common invariant measure}\label{subsubsec:L2-str}

\begin{lemma}[Static disintegration]\label{lem:tractable-static}
In the setting of \Cref{subsubsec:nonlinear-fastslow},
$\pi(dx,dy)=\mu^y(dx)\bar\pi(dy)$. Consequently, if
$f\in\mathcal M$ and $\bar f:=\mathcal Pf\in\bar{\mathcal M}$, then
\begin{equation}\label{eq:tractable-normalization}
m^\eps=\int_Ef\,d\pi=
\int_{\bar E}\bar f\,d\bar\pi=\bar m=:m,\qquad
f_0^\eps=f_0:=f/m,\qquad \bar f_0=\mathcal Pf_0.
\end{equation}
\end{lemma}
\begin{proof}
Disintegrate $\pi(dx,dy)=\pi^y(dx)\widetilde\pi(dy)$. For
$\varphi\in C_c^\infty(\R^{d_x})$ and
$\psi\in C_c^\infty(\bar E)$, invariance gives
\[
\begin{gathered}
0=\int\mathcal L^\eps(\varphi\psi)\,d\pi
=\eps^{-1}A_{\varphi,\psi}+B_{\varphi,\psi},\qquad
A_{\varphi,\psi}:=\int\psi\mathcal L_y^{\rm fast}\varphi\,d\pi,\qquad
B_{\varphi,\psi}:=\int\varphi\mathcal L^{\rm slow}\psi\,d\pi.
\end{gathered}
\]
The two terms are independent of $\eps$; comparison at two values of $\eps$
gives $A_{\varphi,\psi}=B_{\varphi,\psi}=0$. Disintegration, using a
countable subset dense in $C_c^2$ on each compact, and uniqueness of the
frozen invariant measure give $\pi^y=\mu^y$ for $\widetilde\pi$-a.e. $y$.
Taking $\varphi=\chi_R\uparrow1$ in $B_{\varphi,\psi}=0$ and using
$|\mathcal L^{\rm slow}\psi|\le C_\psi(\tr a_2+|b_2|)$ yield
$\int\bar{\mathcal L}\psi\,d\widetilde\pi=0$. By well-posedness of the
conservative martingale problem
\citep[Secs.~6.6 and~7.1--7.2]{StroockVaradhan1979} and Echeverr\'ia's
criterion \citep{Echeverria1982Invariant},
$\widetilde\pi$ is invariant for $\bar P_t$; hence
$\widetilde\pi=\bar\pi$. The remaining identities follow by disintegration.
\end{proof}
\begin{lemma}[Cutoffs for the averaging model]
\label{lem:avg-cutoffs}
In the setting of \Cref{subsubsec:nonlinear-fastslow}, suppose
that, for some $C_{\mathrm{cut}}>0$ and every $z\in E$,
\[
\operatorname{tr}a_1(z)\le C_{\mathrm{cut}}(1+|z|^2),
\qquad
|b_1(z)|\le C_{\mathrm{cut}}(1+|z|).
\]
Then the cutoff condition in
\Cref{ass:standing}\textup{(iii)} holds for both systems.
\end{lemma}

\begin{proof}
Choose smooth radial cutoffs $\chi_R$ on $E$ and $\bar\chi_R$
on $\bar E$, both taking values in $[0,1]$ and increasing to one,
equal to one on $B_R$ and
supported in $B_{2R}$, with first and second derivatives bounded
by $C/R$ and $C/R^2$, respectively.
For each fixed $\eps>0$ and $R\ge1$, the block structure and
the growth bounds give
\begin{align*}
\int_E\bigl(|\mathcal L^\eps\chi_R|
             +\Gamma^\eps(\chi_R)\bigr)\,d\pi
&\le C_\eps\,\pi\{|z|\ge R\}
   +\frac{C}{R^2}\int_E\operatorname{tr}a_2\,d\pi
   +\frac{C}{R}\int_E|b_2|\,d\pi
\longrightarrow0.
\end{align*}
Here the constants are independent of $R$.
By \Cref{lem:tractable-static} and conditional Jensen,
\[
\int_{\bar E}\bigl(\operatorname{tr}\bar A+|\bar b|\bigr)
\,d\bar\pi
\le
\int_E\bigl(\operatorname{tr}a_2+|b_2|\bigr)\,d\pi
<\infty.
\]
Consequently,
\[
\int_{\bar E}\bigl(|\bar{\mathcal L}\bar\chi_R|
                   +\bar\Gamma(\bar\chi_R)\bigr)\,d\bar\pi
\le
\frac{C}{R^2}\int_{\bar E}\operatorname{tr}\bar A\,d\bar\pi
+\frac{C}{R}\int_{\bar E}|\bar b|\,d\bar\pi
\longrightarrow0.
\]
\end{proof}
\begin{lemma}[Smooth test profiles]
\label{lem:avg-test-profiles}
In the setting of \Cref{subsubsec:nonlinear-fastslow}, assume that
$\bar\pi$ has a strictly positive density
$\bar p\in C^\infty(\bar E)$.
Then, for every $h\in C_c^\infty(E)$,
$\mathcal Ph\in C_c^\infty(\bar E)$ and
$c+h\in\mathcal M_{\mathrm{av}}$ whenever $c>\|h\|_\infty$.
\end{lemma}

\begin{proof}
Writing $\pi(dx,dy)=p(x,y)\,dx\,dy$, \Cref{lem:tractable-static} gives
\[
(\mathcal Ph)(y)
=
\frac{1}{\bar p(y)}
\int_{\mathbb R^{d_x}} h(x,y)p(x,y)\,dx
\qquad \bar\pi\text{-a.e.}
\]
The right-hand side is smooth by differentiation under the integral
and is supported in the compact set $\Phi(\operatorname{supp}h)$.
Weak continuity of $y\mapsto\mu^y$ makes $\mathcal Ph$ continuous,
so the identity holds everywhere.

For $c>\|h\|_\infty$, the bound
$\|\mathcal Ph\|_\infty\le\|h\|_\infty$ yields
\[
\sqrt{c+h}-\sqrt c\in C_c^\infty(E),
\qquad
\sqrt{c+\mathcal Ph}-\sqrt c\in C_c^\infty(\bar E).
\]
Constants belong to the respective form domains.
Hence $c+h\in\mathcal M$ and
$\mathcal P(c+h)=c+\mathcal Ph\in\bar{\mathcal M}$,
which proves $c+h\in\mathcal M_{\mathrm{av}}$.
\end{proof}

\begin{lemma}[Gradient-estimate inheritance]
\label{lem:avg-gradient-inheritance}
In the setting of \Cref{subsubsec:nonlinear-fastslow}, assume that
$\bar\pi$ has a strictly positive density
$\bar p\in C^\infty(\bar E)$.
Suppose that the microscopic cutoff condition in
\Cref{ass:standing}\textup{(iii)} holds and that
\Cref{eq:dyn-conv-standing} holds for every
$f\in\mathcal M_{\mathrm{av}}$ and $t>0$, with
$\bar f=\mathcal Pf$.
Then the microscopic clause of \Cref{ass:CDkappa} implies
its macroscopic clause with the same $\kappa$.
\end{lemma}

\begin{proof}
By \Cref{lem:tractable-static}, $\mathcal P$ extends as conditional
expectation from $L^2(\pi)$ to $L^2(\bar\pi)$.
Write $g^\uparrow:=g\circ\Phi$.
For $h\in C_c^\infty(E)$, \Cref{lem:avg-test-profiles} allows us
to apply the assumed orbit convergence to $c+h$, where
$c>\|h\|_\infty$.
Canceling the common normalizer and the constant $c$ gives
local uniform convergence for $h$.
Bounded convergence, density and contractivity then yield
\begin{equation}\label{eq:L2-intertwining}
P_t^\eps h\longrightarrow(\bar P_t\mathcal Ph)^\uparrow
\quad\text{in }L^2(\pi),
\qquad h\in L^2(\pi),\quad t>0.
\end{equation}

Assume the microscopic gradient estimate.
Fix $t>0$ and $\varphi\in C_c^\infty(\bar E)$, and set
\[
r(x,y):=
\sqrt{\nabla\varphi(y)^\top a_2(x,y)\nabla\varphi(y)}.
\]
Since
$\Gamma^\eps(\varphi^\uparrow)=r^2
\le\|\nabla\varphi\|_\infty^2\tr(a_2)\in L^1(\pi)$,
the cutoff condition and the product rule give
$\varphi^\uparrow\in\mathcal D(\mathcal E^\eps)$.
Thus \Cref{lem:gradient-form-extension} gives
\[
|\nabla_yP_t^\eps\varphi^\uparrow|_{a_2}^2
\le e^{2\kappa t}(P_t^\eps r)^2.
\]
Contractivity bounds the slow gradients in $L^2(\pi;a_2)$.
Local ellipticity and \Cref{eq:L2-intertwining} identify their
weak limit as $(\nabla_y\bar P_t\varphi)^\uparrow$.

For $0\le\eta\in C_c^\infty(\bar E)$, weak lower semicontinuity,
disintegration and \Cref{eq:L2-intertwining} applied to $r$ give
\begin{align*}
\int_{\bar E}\eta\,\bar\Gamma(\bar P_t\varphi)\,d\bar\pi
&\le
\liminf_{\eps\downarrow0}
\int_E\eta^\uparrow
|\nabla_yP_t^\eps\varphi^\uparrow|_{a_2}^2\,d\pi
\le e^{2\kappa t}
\int_{\bar E}\eta\,(\bar P_t\mathcal Pr)^2\,d\bar\pi
\le e^{2\kappa t}
\int_{\bar E}\eta\,
\bigl[\bar P_t\sqrt{\bar\Gamma(\varphi)}\bigr]^2\,d\bar\pi.
\end{align*}
The last inequality follows from conditional Jensen:
$\mathcal Pr\le\sqrt{\mathcal P(r^2)}
=\sqrt{\bar\Gamma(\varphi)}$.
Since $\eta$ is arbitrary, the macroscopic gradient estimate
holds $\bar\pi$-a.e.
Local parabolic regularity and the $C_b$-preserving property
of $\bar P_t$ make both sides continuous, so it holds everywhere.
\end{proof}

\subsubsection{Verification of the multiplicative-noise example}
\label{subsubsec:si-gaussian-multiplicative}

\begin{proof}[Proof of \Cref{ex:gaussian-multiplicative}]
\emph{Invariant measure and standing assumptions.}
The identities $\operatorname{div}c_\star=0$ and
$y\cdot c_\star(y)=0$ give
\[
\nabla\cdot\bigl(\gamma^\varepsilon e^{-V^\varepsilon}\bigr)=0.
\]
The symmetric parts are in divergence form with respect to
$\nu_1\otimes\nu_2$ and $\nu_2$, respectively. The coefficients
are globally Lipschitz and have at most linear growth:
$b_1=-x$, while $b_2+y$ and $\bar b+y$ are bounded smooth
functions with bounded derivatives. Thus the SDEs are
nonexplosive and their martingale problems are well posed.
Integration by parts and Echeverr\'ia's criterion identify
the displayed Gaussian laws as invariant for every
$\varepsilon>0$. Uniform ellipticity for each fixed
$\varepsilon$ gives the strong Feller property, irreducibility
and positive-time regularity; in particular, the invariant
measure is unique. The semigroups preserve $C_b$, and invariance,
path continuity and density give strongly continuous
contractions on the corresponding Gaussian $L^2$ spaces.
The frozen fast process is the standard OU process, so
$\mu_y=\nu_1$. Since $\int r_\star\,d\nu_1=1$, averaging gives
the stated macroscopic coefficients. The nonzero irreversible
drifts make both diffusions nonreversible. Gaussian moments
and boundedness of $d_\star,c_\star,r_\star$ give the coefficient
integrability required in \Cref{eq:tractable}.

For the generator-core assertion, let $\mathcal L_{\mathrm c}$
denote the restriction of either differential operator to
$C_c^\infty$, with Gaussian law $\mu$.
It is dissipative by integration by parts. Choose radial
cutoffs $\zeta_R$ with $0\leq\zeta_R\uparrow1$, equal to one
on the ball of radius $R$ and zero outside the ball of radius
$2R$, and $|\nabla\zeta_R|\leq C/R$.
For $w\in\ker(\lambda-\mathcal L_{\mathrm c}^*)$, $\lambda>0$,
local elliptic regularity permits testing with
$\zeta_R^2\overline w$. The irreversible drift annihilates
$\zeta_R$, so taking real parts and using Young's inequality gives
\[
\lambda\|\zeta_Rw\|_{L^2(\mu)}^2
\leq2\int |w|^2\Gamma(\zeta_R)\,d\mu
\leq C_\varepsilon R^{-2}\|w\|_{L^2(\mu)}^2.
\]
Hence $w=0$. The dense-range form of the Lumer--Phillips
theorem makes $\overline{\mathcal L_{\mathrm c}}$ maximally dissipative.
Dynkin's formula identifies it with the diffusion generator,
and $C_c^\infty$ is therefore a core.

The energy forms are equivalent, for each fixed $\varepsilon$,
to the Gaussian Sobolev forms and are closable. The same cutoffs,
chosen with $|D^2\zeta_R|\leq C/R^2$, put $1$ in their domains and give
\[
\int\bigl(|\mathcal L\zeta_R|+\Gamma(\zeta_R)\bigr)\,d\mu
\leq C_\varepsilon\mu\{|z|\geq R\}\longrightarrow0.
\]
This verifies the standing and cutoff assumptions for both
systems. By \Cref{lem:avg-test-profiles},
$\mathcal M_{\mathrm{av}}$ contains $1$ and the shifted test
profiles used in \Cref{lem:avg-gradient-inheritance}.

\emph{The derivative-flow hypotheses.}
We have $1\leq d_\star\leq2$, $B_1=1$,
$B_2=d_\star^{-1}I_2$, $\eta_1=\sqrt2$ and
$\eta_2=\sqrt{2d_\star}I_2$.
Both noise coefficients belong to $C_b^2$, and the first
and second derivatives of $b_1,b_2$ are bounded.
Since the fast noise and metric are constant, $D_y b_1=0$
and $D_x b_2=\cos x\,c_\star(y)$, the constants in
\Cref{eq:IKB-constants} satisfy
\[
\kappa_f=2,\qquad \ell_f=\mu_f=C_h=0,\qquad
\ell_s=\sup_y\frac{|c_\star(y)|}{\sqrt{d_\star(y)}}<\infty.
\]
Outside a fixed compact set in $y$, $d_\star=1$ and $c_\star=0$, so
the normalized slow matrix defining $\mu_s$ is $-2I_2$.
On that compact set it is bounded uniformly in $x$.
Thus $-2\leq\mu_s<\infty$. With
$\delta_f=\delta_s=1/2$,
\[
\Delta=1,\qquad \beta_0=0,\qquad
\beta_s=\mu_s+2\ell_s^2,\qquad
\kappa_{\mathrm{mic}}=(\mu_s/2+\ell_s^2)_+.
\]
This verifies \Cref{ass:IKB-condensed}. Since $(-\beta_s)_+\le 2$, the value of $\varepsilon_0$ specified
in \Cref{ass:IKB-condensed}(4) satisfies $\varepsilon_0\ge 1/4$. By \Cref{thm:CD-negative-short}, the microscopic
gradient estimate holds for $0<\eps\le1/4$ with
$\kappa=\kappa_{\mathrm{mic}}$.

\emph{Averaging and initial profiles.}
In \citet[Thm.~4.4]{BardiKouhkouh2023}, identify the slow
and fast variables with $y$ and $x$, respectively, and use
a singleton control set. With a three-dimensional Brownian
motion, their noise matrices are
\[
\sigma^\varepsilon=\sigma(y)
=\bigl(0_{2\times1},\sqrt{d_\star(y)}I_2\bigr),\qquad
\varrho=(1,0,0).
\]
These matrices satisfy $\sigma\varrho^\top=0$, as required
by the block diffusion matrix in our model.
Their conditions (A1) and (A3) follow from global Lipschitzness,
linear growth and $\varrho\varrho^\top=1$; (A2) is immediate
because $\sigma^\varepsilon=\sigma$. The fast drift is $-x$,
so (C2) holds with contraction constant $1$ and implies (A4),
while (C1) holds with $\bar\varrho=1$.
Take zero running cost and discount rate and terminal payoff
$h\in C_c^\infty(E)$. Then (B1)--(B3) and (C3) hold,
including the uniform Lipschitz condition in the fast variable.
Finally, $\sigma\sigma^\top=d_\star(y)I_2$ has bounded second
derivatives, is independent of the fast variable, and the
fast drift is independent of the slow variable. Thus both
alternatives in (D) hold.
For a fixed duration $t>0$, applying their theorem at
calendar time $T-t$, with $T>t$, gives
\[
P_t^\varepsilon h\longrightarrow
(\bar P_t\mathcal P h)\circ\Phi
\qquad\text{locally uniformly on }E.
\]
Here the effective terminal payoff is $\mathcal P h$, and
the effective generator is the displayed macroscopic generator.

For $f\in\mathcal M_{\mathrm{av}}$, apply this convergence
to $\zeta_R f$ and $\zeta_R$. Both are compactly supported
and smooth. Markov contractivity gives, for every compact $K\subset E$,
\[
\begin{aligned}
&\limsup_{\varepsilon\downarrow0}\sup_{z\in K}
\bigl|P_t^\varepsilon f(z)
-(\bar P_t\mathcal P f)(\Phi(z))\bigr|\leq2\|f\|_\infty\sup_{z\in K}
\bigl[1-(\bar P_t\mathcal P\zeta_R)(\Phi(z))\bigr]
\longrightarrow0.
\end{aligned}
\]
The last limit follows from conservativity, continuity and
Dini's theorem. Since $m^\eps=\bar m>0$, the normalized trajectories converge as required in
\Cref{ass:standing}(ii), with profile class $\mathcal M_{\mathrm{av}}$.
Taking $\Pi=\pi$, the preceding coefficient, cutoff and orbit
arguments verify \Cref{ass:coeff-weak,ass:standing}.
By \Cref{lem:avg-gradient-inheritance}, the microscopic gradient
estimate passes to the macroscopic semigroup with the same constant
$\kappa_{\mathrm{mic}}$, completing the verification of
\Cref{ass:CDkappa}.

\emph{Entropy production.}
The slow force and macroscopic force are
$F_y=r_\star c_\star/d_\star$ and $\bar F=c_\star/d_\star$, and
\[
\int r_\star^2\,d\nu_1=\frac{3-e^{-2}}2,\qquad
v(y)=\operatorname{Var}_{\nu_1}(r_\star)
\frac{|c_\star(y)|^2}{d_\star(y)}
=\frac{1-e^{-2}}2\frac{|c_\star(y)|^2}{d_\star(y)}.
\]
This gives the stated stationary identity. At positive times,
orbit convergence, the common normalization and boundedness
of $r_\star^2|c_\star|^2/d_\star$ give by dominated convergence
\[
\sigma_{\mathrm{hk}}^\varepsilon(t)
\longrightarrow
\int_E\bar u(t,y)r_\star(x)^2
\frac{|c_\star(y)|^2}{d_\star(y)}\,d\pi(x,y)
=\frac{3-e^{-2}}2\,\bar\sigma_{\mathrm{hk}}(t).
\]
Since $c_\star\not\equiv0$ and $\bar u(t,\cdot)>0$, both
$\bar\sigma_{\mathrm{ss}}$ and $\bar\sigma_{\mathrm{hk}}(t)$
are strictly positive. The excess and weighted gradient
conclusions follow from
\Cref{thm:FI-conv,thm:L2-strong}.
\end{proof}

\subsection{Stiff-potential limit}
\label{subsec:stiff-proof}

\begin{proof}[Proof of \Cref{ex:stiff}]
We first verify the basic semigroup properties and
\Cref{ass:standing} under \Cref{eq:stiff-geom}.
Only the final step uses \Cref{eq:stiff-Hessian}.

\emph{Gibbs-measure limit.}
Set
\[
\mathsf L:=D\iota=\binom{H}{I_{d_y}},\qquad
\mathsf C:=(I_{d_x},-H).
\]
Then
\[
\mathsf L^\top\mathsf L=\tilde H,\qquad
\mathsf C\mathsf C^\top=\hat H,\qquad
\mathsf C\mathsf L=0,\qquad
D\Phi=\bar A\mathsf L^\top.
\]
The coordinate map in \Cref{eq:stiff-tubular} satisfies
\[
D\mathcal T=[\mathsf L,\mathsf C^\top \hat H^{-1}],\qquad
(D\mathcal T)^\top D\mathcal T
=\operatorname{diag}(\tilde H,\hat H^{-1}).
\]
Sylvester's determinant identity gives $\det\tilde H=\det \hat H$,
hence $|\det D\mathcal T|=1$.

Put $c:=\frac12\lambda_{\min}(B)-C_1\eps_*^2>0$.
For $0<\eps\le\eps_*$, \Cref{eq:stiff-geom} gives
\begin{equation}\label{eq:stiff-domination}
e^{-V(\mathcal T(\omega,\eps s))-\frac12s^\top Bs}
\le e^{C_0-\Psi(\omega)}e^{-c|s|^2}.
\end{equation}
The right-hand side is integrable, and
$e^{-\bar V}\le e^{C_0-\Psi}$, so $Z_\eps$ and $\bar Z$ are
finite and positive. For $\varphi\in C_b(E)$, the substitution
$z=\mathcal T(\omega,\eps s)$ and dominated convergence yield
\begin{equation}\label{eq:stiff-laplace}
\begin{aligned}
&\eps^{-d_x}\int_E
\varphi(z)e^{-V(z)-\eps^{-2}U(z)}\,dz\\
&\quad=
\int_{\bar E}\int_{\R^{d_x}}
\varphi(\mathcal T(\omega,\eps s))
e^{-V(\mathcal T(\omega,\eps s))-\frac12s^\top Bs}
\,ds\,d\omega\\
&\quad\longrightarrow
c_B\int_{\bar E}\varphi(\iota(\omega))
e^{-\bar V(\omega)}\,d\omega,
\end{aligned}
\end{equation}
where $c_B=(2\pi)^{d_x/2}(\det B)^{-1/2}$.
Dividing by the same expression with $\varphi\equiv1$ gives
$\pi^\eps\rightharpoonup\Pi:=\iota_\#\bar\pi$.
Since $\Phi\circ\iota=\Id_{\bar E}$, we also have
$\Phi_\#\Pi=\bar\pi$, proving
\Cref{ass:standing}\textup{(i)}.

\emph{Semigroup and form properties.}
For each fixed $\eps$, the microscopic drift is locally Lipschitz
and has at most linear growth. The same holds for the macroscopic
drift, since $\nabla\bar V=\mathsf L^\top(\nabla V\circ\iota)$.
Both SDEs are therefore nonexplosive.

We use generic notation
\[
(A_0,V_0,\mu)=
\begin{cases}
(I_{d_x+d_y},V+\eps^{-2}U,\pi^\eps),
&\text{microscopic system},\\
(\bar A,\bar V,\bar\pi),
&\text{macroscopic system}.
\end{cases}
\]
On the corresponding state space $X$, write
$\mu(d\zeta)=p(\zeta)\,d\zeta$. Then
\[
\mathcal Lg=p^{-1}\nabla\!\cdot(pA_0\nabla g),\qquad
\mathcal E_0(g):=
\int_X\nabla g^\top A_0\nabla g\,d\mu,
\qquad g\in C_c^\infty(X).
\]
Since $p$ is positive and $C^2$, the pre-form $\mathcal E_0$
is closable. Choose smooth radial cutoffs $\chi_R$ with
$0\le\chi_R\le1$, $\chi_R\uparrow1$, $\chi_R=1$ on $B_R$,
$\supp\chi_R\subset B_{2R}$,
$|\nabla\chi_R|\le C/R$ and $|\nabla^2\chi_R|\le C/R^2$.
Then
\[
\|\chi_R-1\|_{L^2(\mu)}\longrightarrow0,\qquad
\mathcal E_0(\chi_R)\le CR^{-2}\longrightarrow0.
\]
Thus $1$ belongs to the closed form domain and has zero energy.
The associated symmetric semigroup is conservative and leaves
$\mu$ invariant.

To verify the operator-core property, let $h\in L^2(\mu)$ be
orthogonal to $(\alpha-\mathcal L)C_c^\infty(X)$ for some
$\alpha>0$. Local elliptic regularity gives $h\in H^1_{\rm loc}(X)$.
Testing the weak equation with $\chi_R^2h$ and using Young's
inequality gives
\[
\begin{aligned}
\alpha\int_X\chi_R^2h^2\,d\mu
+\frac12\int_X\chi_R^2\nabla h^\top A_0\nabla h\,d\mu
&\le2\int_X h^2\nabla\chi_R^\top A_0\nabla\chi_R\,d\mu
\le\frac{C}{R^2}\|h\|_{L^2(\mu)}^2.
\end{aligned}
\]
Letting $R\to\infty$ gives $h=0$. The dense-range criterion for
symmetric nonpositive operators therefore makes $C_c^\infty(X)$
a core for the closed-form generator.

Uniqueness of the martingale problems identifies these form
semigroups with the diffusion semigroups. Local ellipticity gives
the strong Feller property and irreducibility, hence uniqueness
of the invariant probabilities. This also verifies the basic
semigroup properties in \Cref{sec:setting}.
In particular, $\gamma^\eps=\bar\gamma=0$, and the preceding
form-domain argument gives $1\in\mathcal M^{\mathrm{stiff}}$.

\emph{Local uniform convergence of the dynamics.}
Fix a compact set $K\subset E$ and $T>0$. For $Z_0^\eps=z\in K$,
set
\[
\Omega_t^\eps:=\Phi(Z_t^\eps),\qquad
r_t^\eps:=\mathsf C\bigl(Z_t^\eps-\iota(0)\bigr).
\]
Projection of \Cref{eq:stiff-nl-SDE} gives
\[
\begin{aligned}
d\Omega_t^\eps
&=-D\Phi\nabla V(Z_t^\eps)\,dt
  +\sqrt2\,\bar A^{1/2}\,d\beta_t,\\
dr_t^\eps
&=-\eps^{-2}\hat HB\,r_t^\eps\,dt
  -\mathsf C\nabla V(Z_t^\eps)\,dt
  +\sqrt2\,\hat H^{1/2}\,d\widetilde\beta_t,
\end{aligned}
\]
where
\[
\beta_t:=\bar A^{1/2}\mathsf L^\top W_t,\qquad
\widetilde\beta_t:=\hat H^{-1/2}\mathsf C W_t.
\]
The matrix identities above show that these are independent
standard Brownian motions. Use the same $\beta$ to couple
$\Omega^\eps$ with
\[
d\bar\Omega_t
=-\bar A\nabla\bar V(\bar\Omega_t)\,dt
 +\sqrt2\,\bar A^{1/2}\,d\beta_t,
\qquad \bar\Omega_0=\Phi(z).
\]
Write $\mathbb P_z$ and $\E_z$ for probability and expectation
under this coupling.

The stiff drift is dissipative relative to $\iota(0)$:
\[
\langle Z_t^\eps-\iota(0),\nabla U(Z_t^\eps)\rangle
=(r_t^\eps)^\top Br_t^\eps\ge0.
\]
Together with the linear growth of $\nabla V$ and $\nabla\bar V$,
It\^o's formula, the Burkholder--Davis--Gundy inequality and
Gronwall's inequality give second-moment bounds for both
processes, uniform in $\eps$ and $z\in K$.

Moreover, $\hat HB$ is similar to $\hat H^{1/2}B\hat H^{1/2}\succ0$, so
$\|e^{-\hat HBs}\|\le Ce^{-\lambda s}$ for some $\lambda>0$ and all
$s\ge0$. Variation of constants gives
\[
\begin{aligned}
r_t^\eps
&=e^{-\hat HBt/\eps^2}r_0
  -\int_0^t e^{-\hat HB(t-s)/\eps^2}
     \mathsf C\nabla V(Z_s^\eps)\,ds+\sqrt2\int_0^t e^{-\hat HB(t-s)/\eps^2}
     \hat H^{1/2}\,d\widetilde\beta_s,
\end{aligned}
\]
where $r_0=\mathsf C(z-\iota(0))$.
By the moment bound, Minkowski's inequality and It\^o's isometry,
the drift and stochastic convolutions are respectively
$O(\eps^2)$ and $O(\eps)$ in $L^2(\mathbb P_z)$, uniformly for
$0\le t\le T$ and $z\in K$. Consequently,
\begin{equation}\label{eq:stiff-fast-est}
\begin{gathered}
\sup_{\substack{z\in K\\0<\eps\le\eps_*}}
\E_z\sup_{0\le s\le T}
\bigl(|Z_s^\eps|^2+|\bar\Omega_s|^2\bigr)
\le C_{T,K},\\
\sup_{z\in K}\E_z|r_t^\eps|^2
\le C_{T,K}\bigl(e^{-2\lambda t/\eps^2}+\eps^2\bigr),
\qquad 0\le t\le T,\\
\sup_{z\in K}\int_0^T\E_z|r_s^\eps|\,ds
\le C_{T,K}(\eps^2+\eps)\longrightarrow0.
\end{gathered}
\end{equation}
Since $\operatorname{dist}(Z_t^\eps,M)=|r_t^\eps|_{\hat H^{-1}}$,
the claimed normal-displacement bound follows.

Let
\[
\tau_R^\eps:=
\inf\{s\ge0:|Z_s^\eps|\vee|\iota(\bar\Omega_s)|\ge R\}.
\]
Since
$D\Phi\nabla V(\iota(\omega))
=\bar A\nabla\bar V(\omega)$,
local Lipschitz continuity of $\nabla V$ and cancellation of the
common noise give, by Gronwall's inequality,
\[
\sup_{s\le T\wedge\tau_R^\eps}
|\Omega_s^\eps-\bar\Omega_s|
\le C_{R,T}\int_0^T|r_s^\eps|\,ds.
\]
The moment bound also gives
\[
\sup_{\substack{z\in K\\0<\eps\le\eps_*}}
\mathbb P_z(\tau_R^\eps\le T)\le C_{T,K}R^{-2}.
\]
Hence, for every $\eta>0$,
\[
\sup_{z\in K}
\mathbb P_z\left(
\sup_{s\le T}|\Omega_s^\eps-\bar\Omega_s|>\eta\right)
\le\frac{C_{T,K}}{R^2}
 +\frac{C_{R,T}}{\eta}
  \sup_{z\in K}\int_0^T\E_z|r_s^\eps|\,ds.
\]
Letting first $\eps\downarrow0$ and then $R\to\infty$ proves
uniform-in-time convergence in probability of the projected
paths, uniformly over $z\in K$. For each fixed $t>0$,
\Cref{eq:stiff-fast-est} and
$Z_t^\eps=\mathcal T(\Omega_t^\eps,r_t^\eps)$ then yield
\[
\sup_{z\in K}
\mathbb P_z\bigl(
|Z_t^\eps-\iota(\bar\Omega_t)|>\eta\bigr)
\longrightarrow0.
\]
Positive time is needed for the initial normal displacement
to decay.

The moment bounds give uniform tightness of both terminal laws.
Uniform continuity on compact sets therefore yields, for every
$h\in C_b(E)$ and $t>0$,
\begin{equation}\label{eq:stiff-pointwise}
\sup_{z\in K}
\left|P_t^\eps h(z)
-\bigl[\bar P_t(h\circ\iota)\bigr](\Phi(z))\right|
\longrightarrow0.
\end{equation}
For $f\in\mathcal A=\mathcal M^{\mathrm{stiff}}$, set
$\bar f=f\circ\iota$. The invariant-measure limit gives
\[
m^\eps=\int_E f\,d\pi^\eps
\longrightarrow
\int_{\bar E}\bar f\,d\bar\pi=\bar m>0.
\]
Normalizing \Cref{eq:stiff-pointwise} with $h=f$ gives
\begin{equation}\label{eq:stiff-normalized-pointwise}
\sup_{z\in K}
\left|P_t^\eps f_0^\eps(z)
-(\bar P_t\bar f_0)(\Phi(z))\right|
\longrightarrow0.
\end{equation}
Thus \Cref{ass:standing}\textup{(ii)} holds, and
$(f,f\circ\iota)$ is admissible for every $f\in\mathcal A$.

\emph{Cutoff condition.}
Use the radial cutoffs introduced above. For each fixed $\eps$,
linear growth of the microscopic drift gives, for $R\ge1$,
\[
\int_E
\bigl(|\mathcal L^\eps\chi_R|+\Gamma^\eps(\chi_R)\bigr)
\,d\pi^\eps
\le C_\eps\,\pi^\eps(\{|z|\ge R\})
\longrightarrow0.
\]
The same argument applies to $\bar{\mathcal L}$ and $\bar\pi$.
The constant may depend on $\eps$, as permitted by
\Cref{ass:standing}\textup{(iii)}.
This completes the first assertion.

\emph{Uniform gradient estimate.}
Assume \Cref{eq:stiff-Hessian}, with $0<\eps<\eps_0$.
Since $\nabla^2U=\mathsf C^\top B\mathsf C$ and
$\mathsf C\mathsf L=0$, restriction to tangent directions gives
\[
\begin{aligned}
\nabla^2\bar V(\omega)
&=\mathsf L^\top
  \nabla^2(V+\eps^{-2}U)(\iota(\omega))\mathsf L
\succeq-\kappa\tilde H=-\kappa\bar A^{-1}.
\end{aligned}
\]
Thus both systems satisfy
$\nabla^2V_0\succeq-\kappa A_0^{-1}$ in the generic notation above.

To obtain the gradient estimate using only $V_0\in C^2$,
consider the diffusion driven by a standard Brownian motion
$\widehat W$:
\[
d\mathsf X_t^\zeta
=-A_0\nabla V_0(\mathsf X_t^\zeta)\,dt
 +\sqrt{2A_0}\,d\widehat W_t,
\qquad \mathsf X_0^\zeta=\zeta.
\]
Subtracting the additive Brownian term reduces this SDE to a
random ODE whose vector field is continuous in time, $C^1$ in
space and of at most linear growth on finite time intervals.
Classical ODE theory therefore gives a global flow that is $C^1$
in the initial point. Its derivative
$J_t^\zeta:=D_\zeta\mathsf X_t^\zeta$ satisfies
\[
\frac{d}{dt}J_t^\zeta
=-A_0\nabla^2V_0(\mathsf X_t^\zeta)J_t^\zeta,
\qquad J_0^\zeta=I.
\]
For every vector $v$,
\[
\begin{aligned}
\frac{d}{dt}|J_t^\zeta v|_{A_0^{-1}}^2
&=-2(J_t^\zeta v)^\top
  \nabla^2V_0(\mathsf X_t^\zeta)J_t^\zeta v
\le2\kappa|J_t^\zeta v|_{A_0^{-1}}^2.
\end{aligned}
\]
Gronwall's inequality yields
$|J_t^\zeta v|_{A_0^{-1}}
\le e^{\kappa t}|v|_{A_0^{-1}}$.
For $g\in C_c^\infty(X)$, this bound justifies differentiation
under the expectation in
$P_tg(\zeta)=\E[g(\mathsf X_t^\zeta)]$.
Duality of the $A_0$ and $A_0^{-1}$ norms gives
\[
\sqrt{\Gamma(P_tg)}(\zeta)
=|\nabla P_tg(\zeta)|_{A_0}
\le e^{\kappa t}P_t|\nabla g|_{A_0}(\zeta)
=e^{\kappa t}P_t\sqrt{\Gamma(g)}(\zeta).
\]
By \Cref{lem:gradient-form-extension}, these estimates extend
to the respective closed form domains with the same $\kappa$.
This proves \Cref{ass:CDkappa}.

The locally uniform convergence of the excess EPR on
$(0,\infty)$ follows from \Cref{cor:F-uniform}\textup{(1)}.
Since $\gamma^\eps=\bar\gamma=0$, the housekeeping terms vanish,
and the same conclusion holds for the total EPR.
\end{proof}

\end{suppinfo}
\end{document}